\documentclass[12pt, a4paper,preprint,review]{article}

\usepackage{graphicx}
\usepackage{lscape}
\usepackage{amsfonts}
\usepackage{amssymb,amsthm,amsmath,color,mathrsfs}
\usepackage[T1]{fontenc}
\usepackage{lmodern}
\usepackage{hyperref}
\usepackage{tikz-cd}
\usepackage{mathtools}

\usepackage{amsmath,amssymb}

\newcommand{\tarc}{\mbox{\large$\frown$}}
\newcommand{\arc}[1]{\stackrel{\tarc}{#1}}
\newtheorem{Th}{Theorem}[section]
\newtheorem{Cor}[Th]{Corollary}
\newtheorem{Lem}[Th]{Lemma}

\theoremstyle{definition}
\newtheorem{Def}[Th]{Definition}
\newtheorem{Not}[Th]{Note}
\newtheorem{Rem}[Th]{Remark}
\newtheorem{Ex}[Th]{Example}
\newtheorem{Nota}[Th]{Notation}

\theoremstyle{plain}

\let\ta\tau 
\renewcommand{\ta}{\scalebox{1.44}{$\tau$}}

\newcommand{\R}{\mathbb R}
\newcommand{\N}{\mathbb N}
\newcommand{\C}{\mathbb C}
\newcommand{\K}{\mathbb K}
\newcommand{\M}{\mathscr M}

\newcommand{\B}{\mathscr B}
\newcommand{\U}{\mathscr U}

\newcommand{\pt}[2]{\mathcal P_{#2}#1}

\begin{document}

\title{Absolute Continuity of Function on Topological Space using Measure}
\author{Dhruba Prakash Biswas\footnote{Department of Pure Mathematics, University of Calcutta, 35 Ballygunge 
		Circular Road, Kolkata-700019, INDIA, 
		e-mail : dhrubaprakash28@gmail.com} and Sandip Jana\footnote{Department of Pure Mathematics, University of Calcutta, 35, Ballygunge Circular Road, Kolkata-700019, INDIA, e-mail : sjpm@caluniv.ac.in}} 

\date{}
\maketitle

\begin{abstract}
The prime objective of this paper is to develop the notion of absolute continuity of functions on a more general setting outside $\R$. For this we have considered a topological space which is a measure space as well. We have built axioms for making the $\sigma$- algebra and measure compatible with the topology  of the space. These spaces are termed as \textit{topological measure space} (in short \textit{tms}). $\R^n$ with usual topology, Lebesgue $\sigma$-algebra and Lebesgue measure is a relevant example of tms. Further, we have presented a new tms structure on second countable metric spaces with the development of a new measure. This construction is motivated by \textbf{Carath\'eodory}'s Theorem. In this new tms framework, we have accomplished exploring ample collection of absolutely continuous functions not only on $\R^n(n\geq 2)$ but also on any seperable normed linear space. Also, we have described several analytical aspects carrying the intrinsic sense of absolute continuity on tms framing. Besides, the collection of all absolutely continuous functions on tms forms a vector space over $\K$, the field of real or complex numbers and with additional boundedness property, they form ring and algebra over $\K$. Thereafter, we have introduced the concept of \textit{locally Lipschitcz function} on tms involving the measurement of open connected sets. A relation between absolute continuity and locally Lipschitz has been developed. We have proved that absolute continuity and boundedness of linear functionals are same on separable normed linear spaces with the association of that new measure. Further, we have extended the co-domain of absolutely continuous functions upto normed linear spaces which helps us to characterise absolute continuity of linear maps in terms of boundedness when the domain is a seperable normed linear space incorporated with that new measure.
\end{abstract}

AMS Classification 2020 : 26A46, 26B30, 28A12, 54B99, 46B99.


\section{Introduction}

 The concept of absolute continuity of functions weighs a stringent importance in real analysis for its elegant properties and versatile applications. It measures the total oscillation of a function (defined on a compact interval) by means of inspecting the functional behaviour over all finite families of disjoint open subintervals of arbitrarily small length. Each subinterval possesses another topological property i.e. connectedness. Accordingly, if we reckon all countable collections each consisting of disjoint open connected fragments of a topological space, can the notion of absolute continuity of functions be generalised ? In \cite{Maly}, Jan Mal\'{y} generalised absolute continuity of functions defined on some open subset of $ \R^n $. In this paper we seek for solutions of the above question and significantly, the answer is affirmative. However, at first a compatiable $\sigma$-algebra and a suitable measure is required to be ensued on the topological space for measurement of sets.
 
  In section 2, we introduce a new topological structure which is a measure space as well; the topology being compatible with its $\sigma$-algebra and existent measure in a reasonable way. This space is named as \textit{topological measure space} (in short \textit{tms}). The $\sigma$-algebra of tms is a refinement of collection of all open sets. As expected, each tms is locally connected due to adequate existence of open connected sets. Further, measurability of those sets turns this space into uniformisable one. $\R^n$ with usual topology, Lebesgue $\sigma$-algebra and Lebesgue measure is a salient example of tms to be contemplated. Unfortunately, the notion of absolute continuity is not satisfactory here because of the presence of zero measurable connected arcs. Even simple curves fail to be tms. The construction of Lebesgue $\sigma$-algebra on $\R^n\,(n\geq 2$) deteres even projection maps and certain trigonometric functions from being absolutely continuous. As a consequence, we require an alternative measure except from Lebesgue measure for better treatment of class of absolutely continuous functions on $ \R^n\,(n\geq 2$). 
  
  \textbf{Carath\'eodory}'s Theorem renders a significant part in constructing tms structure on second countable metric spaces with the help of a new measure. This is discussed in section 3. With respect to this measure, all the rectifiable curves on $\R^2$ constitute a rich class of tms. Existence of ample collection of non-trivial tms envisages relevance and importance of discussion on this structure. This newly manufactured tms framework helps us identify enriched collection of absolutely continuous functions and investigate several analytical aspects conveying the intrinsic sense of absolute continuity on several seperable normed linear spaces. 
  
   The collection of all absolutely continuous functions on tms forms a vector space over $\K$, the field of real or complex numbers and with the additional boundedness property, they form  ring and algebra over $\K$. In the last section, we have developed the concept of \textit{locally Lipschitz function} on tms in terms of measurable open connected sets. A relationship between absolute continuity and locally Lipschitz has been established and this leaves an immediate effect on the residual part of our paper. With the consideration of the aforesaid new measure on seperable normed linear spaces, a cogent interconnection between boundedness and absolute continuity of linear functionals has been built. To be more precise, absolute continuity and boundedness of linear functionals are same on separable normed linear spaces with the association of that referred measure. Similar kind of results hold when the co-domain of absolutely continuous functions is extended upto normed linear spaces. The concepts regarding absolute continuity and boundedness furnish identical classes of linear transformations on any seperable normed linear space incorporated with that new measure. 
 
\section{Topological measure space (tms)}

\begin{Def} Let ($X, \ta $) be a topological space and ($X, \M, m$) be a measure space. Then ($X, \ta, \M, m$) is called a \emph{topological measure space} (in short, \textit{tms}) if the following conditions are satisfied: \\
		 (i) $\M$ contains all open sets of $(X,\ta)$.\\
		 (ii) For every $x \in X$ and every $\varepsilon>0,$ there exists an open connected nbd $U$ of $x$ such that $m(U) <\varepsilon$.\\
		 (iii) For every G $\in \ta$ and for any $x \in G,$ there exists $\varepsilon>0$ such that for all open connected nbd $U$ of $x$ with $m(U)<\varepsilon,$ we have $U \subseteq G$.
	\label{d:tms}
\end{Def}

\begin{Rem} From above Definition \ref{d:tms}, using axioms (iii) and (ii) it follows that every tms is locally connected.
\label{rm:lcon}\end{Rem}

\begin{Ex}	
		(i)  Let $ X =\R$, $\ta_u$ be the usual topology on $\R$, $\mathscr{L}(\R) $ be the Lebesgue $\sigma$-algebra on $\R$, $\lambda$ be the	Lebesgue measure. Lebesgue $\sigma $-algebra on $\R$ contains all open subsets of $\R$. Let $x\in \R$ be arbitrary.  Then for every $\varepsilon> 0,\exists\,$ a basic open connected nbd $V =(x-\frac{\varepsilon}{3}, x + \frac{\varepsilon}{3})$ of $x$ such that $\lambda(V) = \frac{2\varepsilon}{
			3}< \varepsilon.$ Let $G\in\ta_u$  and $x \in G.$ Then there exists $a,b \in \R$ such that $x \in (a,b)\subseteq G$. Choose $0<\varepsilon<\min\{x-a,b-x\}$. Then for every open connected nbd $U$ of $x$ with measure $< \varepsilon, $ we have $U\subseteq (a, b)$ and consequently, $U\subseteq G$. Hence, ($\R, \ta_u,  \mathscr{L}(\R), \lambda$) is a tms.
		
		(ii) ($\R^{n},\ta_u,\mathscr{L}(\R^{n}),\lambda$) is a tms for any $ n\in\N $. Here $ \mathscr{L}(\R^{n})$ is the Lebesgue $ \sigma $-algebra on $ \R^n $ and $\lambda $ denotes the Lebesgue measure on $ \mathscr{L}(\R^{n})$.
		
		(iii) Let $X=\R$, $\ta_{d}$ be the discrete topology on $\R$, $\mathscr A(\ta_d)$ be the smallest $\sigma$-algebra containing all open sets of $(\R,\ta_{d})$. Evidently, $\mathscr A(\ta_d)=\mathscr{P}(\R)$ [the power set of $ \R $]. Let $\lambda_c$ be the counting measure. Then ($\R,\ta_{d},\mathscr{P}(\R),\lambda_c$) is not a tms for the following reason:\\
		Let $a\in\R$ be arbitrary. Then $\{a\}\in\mathscr{P}(\R)$ and $\lambda_c(\{a\}) = 1$. For $\varepsilon=\frac{1}{2}$, there does not exist any open set $V\in\ta_{d}$ containing $ a $ such that $\lambda_c(V)<\frac{1}{2}$.
		
		(iv) Let $X=\R$, $\ta_{d}$ be the discrete topology on $\R$, $\mathscr{L}(\R)$ be the Lebesgue $\sigma$-algebra on $\R$, $\lambda$ be the Lebesgue measure. Then ($\R,\ta_{d},\mathscr{L}(\R),\lambda$) is not a tms, since there are non-measurable open sets in $ (\R,\ta_d) $. So axiom (i) of Definition \ref{d:tms} is violated.
		
		 (v) Let $X:=\{z\in\C :| z|=1\}$, $\ta $ be the subspace topology of $\C$ on $X$, $\mathscr{L}(X)$ be the Lebesgue $\sigma$-algebra on $X$, $\lambda$ be the Lebesgue measure. Then ($X,\ta,\mathscr{L}(X),\lambda$) is not a tms. In fact, if $V:=\{z\in\C:z=e^{i\theta},0<\theta<\frac{\pi}{2}\}$ then $V $  is open in $ X $ and $\lambda(V) = 0$ (since $\lambda(X) = 0$, because $X$ is just an arc in $\C$). Now, if we choose $W:=\{z\in\C:z=e^{i\theta},0<\theta<\pi\}$, then $ W $ is open connected in $ X $ and for every $\varepsilon>0,$ we have $\lambda(W) = 0 <\varepsilon$, but $W\not\subseteq V$. Hence ($X,\ta,\mathscr{L}(X),\lambda$) does not satisfy axiom (iii) of Definition \ref{d:tms} .	
		
			(vi) ($\R,\ta_{l},\B(\R),\lambda$) is not a topological measure space \big[here $\ta_{l} $ is the lower limit topology on $\R$, $\B(\R)$ denotes the Borel $\sigma$-algebra on $ (\R,\ta_l) $, $\lambda$ denotes the Lebesgue measure on $ \R $\big] because singletons are the only connected sets and these sets cannot be open. So axiom (ii) of Definition \ref{d:tms} is not satisfied.
		
		(vii) Let $X=\R$, $\ta_f$ be the cofinite topology on $\R$, $\mathscr{A}(\ta_f)$ be the $\sigma$-algebra generated by all open subsets of ($\R,\ta_f$) and $\lambda$ be the Lebesgue measure. Then ($\R,\ta_f,\mathscr{A}(\ta_f),\lambda$) is not a tms, since for any non-empty open set $U$, we have $\lambda(U)$ is infinite, i.e. ($\R,\ta_f,\mathscr{A}(\ta_f),\lambda$) violates axiom (ii) and (iii) of Definition \ref{d:tms}. 
		
		For similar reason, ($\R,\ta_{co},\mathscr{A}(\ta_{co}),\lambda$) is not a tms where $\ta_{co}$ is the co-countable topology on $\R$.	
	\label{e:tms}
\end{Ex}
\begin{Not}
	Example (iii) of \ref{e:tms} shows that, any topological space with counting measure $ \lambda_c $ cannot be a topological measure space, since for any $\varepsilon<1,\not\exists$ any non empty set $E$ such that $\lambda_c(E)<\varepsilon$.
\end{Not}
\begin{Rem}
	From example \ref{e:tms} (v), we can conclude that any simple curve on $\R^n\,(n\geq 2)$ is not a tms with respect to Lebesgue measure.
\end{Rem}
\begin{Lem}
	Let $(X,\ta)$ be a topological space and $Y$ be a subspace of $X$. Then every connected subset of $(Y,\ta_Y)$ is also connected in $(X,\ta)$, where $\ta_Y$ is the subspace topology of $\ta$ onY. \label{l:con}
\end{Lem}

	\begin{proof}
	Let $E\subseteq Y$ be a connected set in $Y$. If possible, let $ E $ be disconnected in $ X $. Then there are disjoint open sets $ U,V $ in $ X $ such that $ U\cup V=E $. So $ U,V $ being subsets of $ Y $, form a disconnection of $ E $ in $ Y $. This contradicts that $ E $ is connected in $ Y $.
\end{proof}
	
\begin{Th}
	Let $(X,\ta,\M,m)$ be a tms and $Y$ be an open subspace of $X$. Then $Y$ is a tms.\label{t:opentms}
\end{Th}

\begin{proof}Define $\M_{Y}\coloneqq\{A\cap Y: A\in \M\}$. Clearly it is a $ \sigma $-algebra in $ Y $. For this, it should only be noted that for any $ A\in\M, Y\smallsetminus A\cap Y=(X\smallsetminus A)\cap Y $. Let $ m_Y $ be the restriction of $ m $ to $ \M_Y $. Since $Y$ is an open subspace of $X$, every open set of $Y$ is also open in $(X,\ta)$ and hence is $ \M $-measurable. Thus, $\M_Y$ contains all open sets of $Y$.
	
	Now let $y\in Y$ be arbitrary. Thus $y\in X$ and for every $\varepsilon>0,\exists$ an open connected nbd $W$ of $y$ in $X$ such that $m(W)<\varepsilon$. $ X $ being tms is locally connected, by Remark \ref{rm:lcon}. So $ Y $ being an open subspace of $ X $ is a locally connected subspace of $X$. Consequently, $\exists$ an open connected subset $U$ of $Y$ such that $y\in U\subseteq W\cap Y$ and hence $m_{Y}(U)\leq m_{Y}(W\cap Y)\leq m(W)< \varepsilon$. 
	
	Now let $V$ be any open subset of $Y$ and $v\in V$ be arbitrary. Then $V$ is an open nbd of $v$ in $X$. Consequently, $\exists$ an $\varepsilon> 0$ such that for every open connected nbd $U$ of $v$ in $X$ with $m(U)< \varepsilon,$ we have $U\subseteq V\cdots\cdots(1)$ [Since $X$ is a tms]. Since $Y$ is an open subspace of $X,$ every open connected subset of $Y$ is also an open connected subset of $X$ [by Lemma \ref{l:con}] . Therefore from $(1)$, for every open connected nbd $W$ of $v$ in $Y$ with $m_{Y}(W)<\varepsilon$, we have $ m(W)=m_Y(W)<\epsilon $ and hence $W\subseteq V$. Hence $Y$ is a topological measure space.   
\end{proof}

Looking at the above theorem we can define the following concept.

\begin{Def}
	Let $(X,\ta,\M,m)$ be a tms and $Y\subset X$. Let $ \ta_Y $ be the subspace topology on $ Y $, $\M_{Y}:=\{A\cap Y: A\in \M\}$ and $ m_Y $ be the restriction of $ m $ to $ \M_Y $. If $ (Y,\ta_Y,\M_Y,m_Y) $ is a tms then it is called a \textit{sub topological measure space} or \textit{subtms} of $ X $.
\label{d:subtms}\end{Def}

\begin{Th}
$ \big([a,b],\ta_{ab},\mathscr{L}([a,b]),\lambda\big) $ is a subtms of the tms $(\R,\ta_{u},\mathscr{L}(\R),\lambda)$, where $ \ta_{ab} $ is the subspace topology on $ [a,b] $ induced from the real line $ \R $, $ \mathscr{L}([a,b]) $ is the Lebesgue $ \sigma $-algebra on $ [a,b] $ and $ \lambda $ is the Lebesgue measure. \label{t:[a,b]tms}
\end{Th}

\begin{proof}
	Clearly $\mathscr{L}([a,b]) = \big\{A\cap[a,b]: A\in\mathscr{L}(\R)\big\}$ contains all open subsets of $[a,b]$, $ \mathscr{L}(\R) $ being the Lebesgue $ \sigma $-algebra on $ \R $. 
	
	Now for each $x\in[a,b]$ and for every $\varepsilon\in(0,b-a),\exists$ an open connected nbd $V=(x-\frac{\varepsilon}{3}, x+\frac{\varepsilon}{3})$ of $x$ such that $\lambda
	(V)=\frac{2\varepsilon}{3}<\varepsilon$. 
	
	Thereafter, let $W$ be any open subset of $[a,b]$ and $w\in W$. Then $\exists$ a basic open subset $V$ of $[a,b]$ such that $w\in V \subseteq W$.\\
	\textbf{Case 1:}  Let $V = (c,d), a\leq c< w< d\leq b$. Choose $0<\varepsilon< \min\{w-c,d-w\}$. Then every open connected nbd $U$ of $w$ with measure $<\varepsilon$ is always a subset of $ V $ and consequently, $U$ is a subset of $W$.\\
	\textbf{Case 2:} Let $V = [a,c),$ for some $c\leq b$. Then $a\leq w< c$. If $w = a,$ then choose $0< \varepsilon< c-a$. Thus any open connected nbd $U$ of $w$ with $\lambda(U)<\varepsilon$ is always of the form [$a, a+\varepsilon'$) where $0<\varepsilon'\leq\varepsilon$. Therefore $a\in U = [a,a+\varepsilon')\subseteq V\subseteq W.$ Now if $w\neq a,$ then choose $0<\varepsilon<\min \{w-a,c-w\}$. Then every open connected nbd of $w$ with measure $<\varepsilon$ is always a subset of $V=[a,c)$ and consequently, a subset of $W$.\\
	\textbf{Case 3:} Let $V=(c,b],$ for some $c\geq a$. Then $c< w\leq b$. The rest of proof of this case is similar as Case 2.\\
	\textbf{Case 4:} Let $V=[a,b].$ Then $a\leq w\leq b$. This case follows from Case 1, 2 and 3.
	
	Therefore, $\big([a,b],\ta_{ab},\mathscr{L}([a,b]),\lambda\big)$ is a sub topological measure space of $(\R,\ta_{u},\mathscr{L}(\R),\lambda)$. 
\end{proof}

\begin{Cor}
	Every connected subset of $\R$ is a subtms of $(\R,\ta_{u},\mathscr{L}(\R),\lambda)$. \label{c:contms}
\end{Cor}

\begin{Th}
	Every tms is uniformisable.
\label{t:unif}\end{Th}

\begin{proof}
	 Let $(X,\ta,\M,m)$ be a tms. For every $\varepsilon>0,$ we define\\ \centerline{$N(\varepsilon):=\big\{(x,y)\in X\times X: x,y\in U\text{ for some open connected set }U\text{ with } m(U)<\varepsilon\big\}$} and $\B(X) := \{N(\varepsilon) : \varepsilon>0\}$.\\ 
	 \textbf{\underline{Step 1}:} \textit{$\B(X)$ forms a base for some uniformity on $X$.}
	 
	(a) Let $x\in X$.  Then for every $\varepsilon>0, \exists$ an open connected  set $U$ in $X$ such that $x\in U$ and $m(U)<\varepsilon$ (by axiom (ii) of Definition \ref{d:tms} ). Thus $(x,x)\in N(\varepsilon),\forall\varepsilon>0$. Consequently, $\varDelta(X):=\big\{(x,x)\in X\times X: x\in X\big\}\subseteq N(\varepsilon)$ for all $\varepsilon>0$.
	
	(b)  $(x,y)\in[N(\varepsilon)]^{-1}\iff(y,x)\in N(\varepsilon)\iff\exists$ an open connected set $U\ni y,x$ and $m(U)<\varepsilon\iff(x,y)\in N(\varepsilon)$. Therefore $N(\varepsilon) = [N(\varepsilon)]^{-1},\forall\,\varepsilon>0.$
	
	(c) Let $N(\varepsilon)\in\B(X)$ for some $\varepsilon>0$. Let $(x,y)\in N(\frac{\varepsilon}{2})\circ N(\frac{\varepsilon}{2}) \implies\exists\, z\in X,$ such that $(x,z)\in N(\frac{\varepsilon}{2}),(z,y)\in N(\frac{\varepsilon}{2})\implies\exists$ two open connected sets $U$ and $V$ such that $x,z\in U$, $z,y\in V$, $m(U)<\frac{\varepsilon}{2}$ and $m(V)<\frac{\varepsilon}{2}$. Now $W:=U\cup V$ is an open connected set (since $U$ and $V$ both are connected and $z\in U\cap V$) such that $x,y\in W$ and $m(W)\leq m(U)+m(V)<\varepsilon\,\implies\,(x,y)\in N(\varepsilon)$. Therefore $N(\frac{\varepsilon}{2})\circ N(\frac{\varepsilon}{2})\,\subseteq\,N(\varepsilon)$ for every $\varepsilon>0$.
	
	(d) Let $N(\varepsilon), N(\delta)\in\B(X)$. Choose $\eta>0$ such that $0<\eta< \min\{\varepsilon,\delta\}$.
    Now, $(x,y)\in N(\eta)$  $\implies\exists$ an open connected  set $U$ such that $x,y\in U,m(U)<\eta\implies(x,y)\in N(\varepsilon)$, $(x,y)\in N(\delta)$ (since $\eta<\varepsilon,\eta<\delta)\implies(x,y)\in N(\varepsilon)\cap N(\delta)$. Thus $N(\eta)\subseteq N(\varepsilon)\cap N(\delta)$.\\ Hence, $\B(X)$ forms a base for some uniformity $\U$ on $X$.\\
    \textbf{\underline{Step 2}:} \textit{If $\ta_{\U}$ is the topology generated by $\U$ on $X$, then $\ta_{\U}=\ta$.} \\   A base for the topology $\ta_{\U}$ on $X$ is given by $\mathscr{N}:=\big
    \{N(\varepsilon)[x] : x\in X, \varepsilon>0\big\}$ where $N(\varepsilon)[x]:= \{y\in X:(x,y)\in N(\varepsilon)\}$. Let $G\in\ta_{\U}$ and $x\in G$. Then $\exists\,\varepsilon>0$ such that $x\in N(\varepsilon)[x]\subseteq G\implies(x,x)\in N(\varepsilon)\implies\exists$ an open connected set $U$ in $(X,\ta)$ containing $x$ and $m(U)<\varepsilon$. We claim that $U\subseteq G$. For this, let $y\in U$. Then  $x,y\in U$ and $m(U)<\varepsilon\implies(x,y)\in N(\varepsilon)\implies y\in N(\varepsilon)[x]\subseteq G\implies U\subseteq G\implies G\in\ta$. Therefore $\ta_{\U}\subseteq\ta$. 
    
    Conversely, let $V\in\ta$ and $x\in V$. Then, by axiom (iii) of Definition \ref{d:tms} , $\exists\,\varepsilon>0$ such that for every open connected nbd $U$ of $x$ with $m(U)<\varepsilon$, we have $U\subseteq V\cdots\cdots$(i). Let $y\in N(\varepsilon)[x]$. Then $(x,y)\in N(\varepsilon)\implies\exists$ an open connected set $U$ in $(X,\ta)$ such that $x,y\in U$ and $m(U)<\varepsilon\implies U\subseteq V$ [by (i)] $\implies y \in V$. Thus, $x\in N(\varepsilon)[x]\subseteq V\implies V\in\ta_{\U}$. Therefore,$\ta\subseteq\ta_{\U}$. Consequently, $\ta=\ta_{\U}$.
   
    Therefore, $X$ is uniformisable. 
\end{proof}

\begin{Th}
	Let $(X,\ta,\M,m)$ be a tms with the additional property : for each $ x,y\in X$, $\exists $ an open connected set $ U $ in $ (X,\ta) $ such that $ x,y\in U $ and $ m(U)<\infty $. Then the tms is pseudometrisable.
\label{t:pseudo}\end{Th}

\begin{proof}
Define $d:X\times X\rightarrow \R$ by\\ \centerline{$d(x,y):=\inf\big\{m(U): x,y\in U,U\text{ open connected in }(X,\ta)\big\}$}
\noindent Then, by hypothesis $ d $ is well-defined. Also clearly,\\ (i) $d(x,y)\geq 0$ for all $x,y\in X$\\ (ii) $d(x,y)=d(y,x)$ for all $x,y\in X$\\ (iii) Let $x\in X$. Since $X$ is a tms, for every $\varepsilon>0$, $\exists$ an open connected set $V$ containing $x$ with $m(V)<\varepsilon$. Thus $\inf\{m(U): x\in U,U\text{ open connected}\}<\varepsilon$ and it implies that $ d(x,x)<\varepsilon$. Since $\varepsilon>0$ is arbitrary, $d(x,x)=0$. Hence $x=y\implies d(x,y)=0$.\\ (iv) Let $x,y,z\in X$. Then for any $ \varepsilon>0, \exists $ open connected sets $ V,W $ such that $ x,z\in V$, $y,z\in W $ and $ m(V)<d(x,z)+\varepsilon,\:m(W)<d(z,y)+\varepsilon $. Now $ V\cup W $ is an open connected set containing $ x,y $ [since $ z\in V\cap W $]. So $ m(V\cup W)\leq m(V)+m(W)<d(x,z)+\varepsilon+d(z,y)+\varepsilon\Rightarrow $ $ d(x,y)\leq m(V\cup W)<d(x,z)+d(z,y)+2\varepsilon $. Since $ \varepsilon>0 $ is arbitrary, it follows that $d(x,y)\leq d(x,z)+d(z,y),\forall\, x,y,z\in X$.

Hence $(X,d)$ is a pseudometric space.\\ Now for every $\varepsilon>0,U_{\varepsilon}:=\{(x,y)\in X\times X:d(x,y)<\varepsilon\}$ is an entourage of $(X,d)$. Now $d(x,y)<\varepsilon\Longleftrightarrow\exists$ an open connected set $U$ containing $x,y$  with $m(U)<\varepsilon$. Hence $U_{\varepsilon}=\{(x,y)\in X\times X :x,y\in U\text{ for some open connected set }U\text{ with }m(U)<\varepsilon\}=N(\varepsilon)$ for all $\varepsilon>0$ [$N(\varepsilon)$ is defined as in Theorem \ref{t:unif}]. Thus, the family $\mathscr{B}(X)=\{N(\varepsilon): \varepsilon>0\}$ is precisely the collection of all entourages $\{U_{\varepsilon}:\varepsilon>0\}$ of $(X,d)$. Since $(X,\ta)$ is uniformisable ($ \mathscr{B}(X) $ being a base for the uniformity) by Theorem \ref{t:unif}, it then follows that $(X,\ta,\M,m)$ is pseudometrisable.
\end{proof}

\begin{Rem}
	If we start with a Hausdorff topological measure space $ (X,\ta,\M,m) $ then it will be metrisable, since then the pseudo-metric $ d $ defined in Theorem \ref{t:pseudo}, becomes a metric. In fact, if $ x,y\in X $ with $ x\neq y $, then $ \exists $ disjoint open sets $ U,V $ containing $ x,y $ respectively. Now $ (X,\ta) $ being pseudo-metrisable, $ \exists\,r>0 $ such that $ B(x,r):=\{y\in X:d(x,y)<r\}\subseteq U $. So $ y\notin U\implies d(x,y)\geq r>0 $.
\end{Rem}

\section{Construction of tms using metric}

In this section, we construct a measure on second countable metric space and turn it into a tms with the help of \textbf{Carath\'eodory}'s Theorem \ref{t:car} . Subsequently, we have presented that each separable normed linear space forms a tms with the aid of this new measure. Further, we have developed a tms structure on $S^1$ with the help of smallest arc length metric and therefore, rectifiable curves of $\R^2$ constitute a rich class of tms.
\begin{Def}\cite{Di}
	Let $X$ be a set and $\nu:2^X\rightarrow[0,\infty]$ be an outer measure. A subset $A\subseteq X$ is called \emph{$\nu$-measurable} if it satisfies $\nu(D)=\nu(D\cap A)+\nu(D\smallsetminus A)$ for every subset $D\subseteq X$. [Here $ 2^X $ denotes the set of all subsets of $ X $]
\end{Def}

\begin{Th}\textbf{\em{(Carath\'eodory)}}{\em\cite{Di}}
	Let $X$ be a set, let $\nu:2^X\rightarrow[0,\infty]$ be an outer measure and define $\mathscr{A}:=\mathscr{A}(\nu):=\{ A\subseteq X:A\text{ is } \nu\text{-measurable }\}$.
	 Then $\mathscr{A}$ is a $\sigma$-algebra, the function $\mu:=\nu|_\mathscr{A}:\mathscr{A}\rightarrow[0,\infty]$ is a measure and the measure space $(X,\mathscr{A},\mu)$ is complete. 
\label{t:car}\end{Th}

\begin{Th}\textbf{\em{(Carath\'eodory Criterion)}}{\em\cite{Di}}
	Let $(X,d)$ be a metric space and $\nu:2^X\rightarrow[0,\infty]$ be an outer measure. Let $\mathscr{A}(\nu)\subseteq2^X$ be the $\sigma$-algebra defined in Theorem \ref{t:car} and let $\mathscr{B}\subseteq 2^X$ be the Borel $\sigma$-algebra of $(X,d)$. Then the following are equivalent:\\
	(i) $\mathscr{B}\subseteq\mathscr{A}(\nu)$.\\
	(ii) If $A,B\subseteq X$ satisfy $d(A,B):=\displaystyle \inf_{a\in A,b\in B}d(a,b)>0$, then $\nu(A\cup B)=\nu(A)+\nu(B)$. \label{t:car2}
\end{Th}

\begin{Th}
	Let $(X,d)$ be a 2nd countable metric space and $\nu:2^X\rightarrow[0,\infty]$ be a function defined by $\nu(A):=\displaystyle\inf \left\{\sum_{n\in\N} \text{diam }(B_n):A\subseteq\bigcup_{n\in\N}B_n,\,B_n \text{ open}, \forall\, n\in\N\right\}$. Then $\nu$ is an outer measure.\label{t:outer}
\end{Th}

\begin{proof}
	For any $A\subseteq X,\nu(A)\in [0,\infty]$. Now, we prove the theorem in the following three steps.\\
	\textbf{Step 1:} If $A_1\subseteq A_2$, then $\nu(A_1)\leq\nu(A_2)$. In fact, 
	$A_1\subseteq A_2$ implies any open cover of $A_2$ is an open cover of $A_1$ also. Hence $\nu(A_1)\leq \nu(A_2)$.\\
	\textbf{Step 2:} \textit{If $A_i\subseteq X,\forall\, i\in\N$, then $\displaystyle\nu\left(\bigcup_{i\in\N}A_i\right)\leq\sum_{i\in\N}\nu(A_i)$.}\\ To prove this, let us denote $A=\displaystyle\bigcup_{i\in\N}A_i$. If there exists a $j\in\N$ such that $\nu(A_j)=\infty$, then obviously $\nu(A)\leq\displaystyle\sum_{i\in\N}\nu(A_i)$. Let $\nu(A_i)<\infty$ for all $i\in\N$. Then for each $i\in\N$ and $\forall\, \varepsilon>0,$ there exists a sequence of open sets $\{B_{n}^{i}\}_{n}$ covering $A_i$, such that $\displaystyle\sum_{n\in\N}\text{diam }(B_{n}^{i})<\nu(A_i) + \tfrac{\varepsilon}{2^i}$\ for all $i\in\N$  $\implies$  $\displaystyle\sum_{i\in\N}\sum_{n\in\N}^{} \text{diam}(B_{n}^{i})\leq\sum_{i\in\N}[\nu(A_i)+\tfrac{\varepsilon}{2^i}]=\sum_{i\in\N}\nu(A_i)+\varepsilon$. Since $\{B_{n}^{i}\}_{i,n}$ is a countable open cover of $A$, therefore by definition of $\nu$, $\nu(A)\leq\displaystyle\sum_{i\in\N}\nu(A_i)+\varepsilon$. Since $\varepsilon>0$ is arbitrary, $\nu(A)\leq\displaystyle\sum_{i\in\N}\nu(A_i)$. \\
	\textbf{Step 3:} \emph{$\nu(\emptyset)=0$.} To prove this let $\varepsilon>0$ be arbitrary and $\{x_n\}_n$ be any sequence of points on $X$. Then $\emptyset\subseteq\displaystyle\bigcup_{n\in\N}B(x_n,\tfrac{\varepsilon}{2^n})\implies$  $\nu(\emptyset)\leq\displaystyle\sum_{n\in\N}\text{diam }B(x_n,\tfrac{\varepsilon}{2^n}) =\displaystyle\sum_{n\in\N}\tfrac{\varepsilon}{2^{n-1}}=2\cdot\varepsilon$.
	Since $\varepsilon>0$ is arbitrary
	, $\nu(\emptyset)=0$. \\ 
	Hence $\nu$ is an outer measure.
	\end{proof}

\begin{Rem}
	Let $(X,d)$ be a 2nd countable metric space and $\nu:2^X\rightarrow[0,\infty]$ be an outer measure as defined in Theorem \ref{t:outer}. Then by Theorem \ref{t:car} , $(X,\mathscr{A},\mu)$ is a complete measure space where $\mathscr{A}$ and $\mu$ are defined as in Theorem \ref{t:car} . \label{r:outer3} 
\end{Rem}

\begin{Th}
	Let $(X,d)$ be a 2nd countable metric space and $\nu:2^X\rightarrow[0.\infty]$ be an outer measure defined as in Theorem \ref{t:outer} . If $A,B\subseteq X$ satisfy $d(A,B):=\displaystyle \inf_{a\in A,b\in B}d(a,b)>0$, then $\nu(A\cup B)\geq\nu(A)+\nu(B)$.\label{t:outer2}
\end{Th}

\begin{proof}
	Let $\{B_n\}_{n\in\N}$ be an open cover of $A\cup B$. Let us denote $\varLambda_1:=\{n\in\N:B_n\cap B=\emptyset\}$, $\varLambda_2:=\{n\in\N:B_n\cap A=\emptyset\}$ and $\varLambda_3:=\{n\in\N:B_n\cap A\not=\emptyset,B_n\cap B\not=\emptyset\}$. Choose $n\in\varLambda_3$ and fix it. Then $B_n\cap A\not=\emptyset$ and $B_n\cap B\not=\emptyset$. Let $a\in B_n\cap A,b\in B_n\cap B$. So $\text{diam }(B_n)\geq d(a,b)\geq d(A,B)=:\varepsilon>0$ (by hypothesis).\\	
	 Let us define $S_{A}^{n}:=\{x\in B_n:d(x,A)\geq\frac{\varepsilon}{4}\}$. Then $A\cap S_{A}^{n}=\emptyset$. Let $x,y\in B_n\smallsetminus(\overline{B}\cup S_{A}^{n})\implies x,y\in B_n$ but $x,y\not\in(\overline{B}\cup S_{A}^{n})\implies x,y\not\in\overline{B}$ and $x,y\not\in S_{A}^{n}\implies d(x,A)<\frac{\varepsilon}{4}$,  $d(y,A)<\frac{\varepsilon}{4}$. Now for any $a\in A,\,d(x,y)\leq d(x,a)+d(y,a)\implies d(x,y)\leq d(x,A)+d(y,A)<\frac{\varepsilon}{2}$. Thus $x,y\in B_n\smallsetminus(\overline{B}\cup S_{A}^{n})\implies d(x,y)<\frac{\varepsilon}{2}$. Since $x,y\in B_n\smallsetminus(\overline{B}\cup S_{A}^{n})$ are arbitrary, $\text{diam }(B_n\smallsetminus(\overline{B}\cup S_{A}^{n}))\leq\frac{\varepsilon}{2}$ ...(i).\\ Let $S_{B}^{n}:=\{x\in B_{n}:d(x,B)\geq\frac{\varepsilon}{4}\}$. Then $S_{B}^{n}\cap B=\emptyset$ and similarly as in (i), $\text{diam } (B_n\smallsetminus(\overline{A}\cup S_{B}^{n}))\leq\frac{\varepsilon}{2}$ ...(ii).\\ Hence from (i) and (ii), for any $n\in\varLambda_3$, we obtain $\text{diam } (B_n)\geq\varepsilon=\frac{\varepsilon}{2}+\frac{\varepsilon}{2}\geq\text{ diam }(B_n\smallsetminus(\overline{B}\cup S_{A}^{n}))+\text{diam}(B_n\smallsetminus(\overline{A}\cup S_{B}^{n}))\implies$  $\displaystyle\sum_{n\in\varLambda_3}\text{diam } (B_n)\geq\sum_{n\in\varLambda_3}\text{diam} (B_n\smallsetminus(\overline{B}\cup S_{A}^{n}))+\sum_{n\in\varLambda_3} \text{diam} (B_n\smallsetminus(\overline{A}\cup S_{B}^{n}))$. Now $\{B_n:n\in\varLambda_1\}\cup\{B_n\smallsetminus(\overline{B}\cup S_{A}^{n}):n\in\varLambda_3\}$ is an open cover of $A$ and $\{B_n:n\in\varLambda_2\}\cup\{B_n\smallsetminus(\overline{A}\cup S_{B}^{n}):n\in\varLambda_3\}$ is an open cover of $B$. Hence $\displaystyle\sum_{n}\text{diam} (B_n)=\sum_{n\in\varLambda_1}\text{diam} (B_n) + \sum_{n\in\varLambda_2} \text{diam} (B_n) + \sum_{n\in\varLambda_3}\text{diam} (B_n)\geq\sum_{n\in\varLambda_1}\text{diam} (B_n) + \sum_{n\in\varLambda_3}\text{diam} (B_n\smallsetminus(\overline{B}\cup S_{A}^{n}))+\sum_{n\in\varLambda_3}\text{diam} (B_n\smallsetminus(\overline{A}\cup S_{B}^{n}))+\sum_{n\in\varLambda_2}\text{diam} (B_n)\geq\nu(A)+\nu(B)$. Thus for any countable open cover $\{B_n\}_{n\in\N}$ of $A\cup B$, we have $\displaystyle\sum_{n\in\N}\text{diam} (B_n)\geq\nu(A)+\nu(B)$. Hence by definition of $\nu,\;\nu(A\cup B)\geq \nu(A)+\nu(B)$.
\end{proof}

\begin{Rem}
	Let $(X,d)$ be a 2nd countable metric space and $\nu$ be an outer measure as defined in Theorem \ref{t:outer}. Then from Theorems \ref{t:car}, \ref{t:car2}, \ref{t:outer2} and Remark \ref{r:outer3} it follows that, $(X,\mathscr{A},\mu)$ is a complete measure space and $\mathscr{A}$ contains the Borel $\sigma$-algebra $\mathscr{B}$ on $X$. \label{rm:complmsp}
\end{Rem}

\begin{Th}
	Let $(X,d)$ be a second countable metric space in which every open ball is connected. Then $(X,d,\mathscr{A},\mu)$ is a tms, where $\mathscr{A}$ and $\mu$ are as mentioned in Theorems \ref{t:car} and \ref{t:outer}. 
\label{t:mtms}\end{Th}

\begin{proof}
	For the metric space $(X,d)$, the $\sigma$-algebra $\mathscr{A}$ contains all of its open sets (follows from Remark \ref{rm:complmsp} ). Thus $(X,d,\mathscr{A},\mu)$ satisfies axiom (i) of Definition \ref{d:tms} . 
	
	Since every open ball of $(X,d)$ is connected (by hypothesis), for each $x\in X$ and for each $\varepsilon>0$ there exists an open connected set $B(x,\frac{\varepsilon}{3})$ containing $x$ such that $\mu(B(x,\frac{\varepsilon}{3}))=\text{diam}(B(x,\frac{\varepsilon}{3}))=2\cdot \frac{\varepsilon}{3}<\varepsilon$. Therefore $(X,d,\mathscr{A},\mu)$ satisfies axiom (ii) of Definition \ref{d:tms} . 
	
	Let $G$ be any open set and $x\in G$. Then $\exists\,\varepsilon>0$, such that $x\in B(x,\varepsilon)\subseteq G$. Let $V$ be any open connected set containing $x$ such that $\mu(V)<\varepsilon$. Let $y\in V$ be arbitrary. Then $d(x,y)\leq\mu(V)<\varepsilon\implies y\in  B(x,\varepsilon)$. Hence $V\subseteq B(x,\varepsilon)\subseteq G$. Consequently, $(X,d,\mathscr{A},\mu)$ satisfies axiom (iii) of Definition \ref{d:tms} .
	 
	 Hence, $(X,d,\mathscr{A},\mu)$ is a tms.
	 \end{proof}

\begin{Rem}
	In a normed linear space, every open ball being convex is path-connected and hence connected. Thus every 2nd countable (or equivalently, separable) normed linear space is a tms.
\label{rm:nlstms}\end{Rem}

\begin{Not}
	Let $ S^1:=\big\{(a,b)\in\R^2:a^2+b^2=1\big\} $ be the unit circle in the Euclidean plane $ \R^2 $. Define $d:S^1\times S^1\rightarrow\R$ by  $d(x,y):=$the length of smallest arc $ \arc{(xy)} $ along $ S^1 $ joining $ x,y $, for all $x,y\in S^1$. Then $ d $ is a metric on $S^1$; for convenience let us produce a simple justification.\\(i) $d(x,y)\geq0$ for all $x,y\in S^1$ and $d(x,y)=0\iff x=y$.\\ (ii) $d(x,y)= \text{arclength} \arc{(xy)}= \text{arclength} \arc{(yx)}=d(y,x)$, for all $x,y\in S^1$.\\ (iii) $d(x,z)=d(x,y)+d(y,z)$, if $y$ lies inside the smallest arc $\arc{xz}$ and $d(x,z)\leq d(x,y)+d(y,z)$, if $y$ lies outside the smallest arc $\arc{xz}$. Thus $d(x,z)\leq d(x,y)+d(y,z),\forall\, x,y,z\in S^1$.\label{n:metric} 
\end{Not}

\begin{Not} Let $ d_2 $ denotes the subspace metric on $ S^1 $ inherited from the Euclidean plane $ \R^2 $. Then the only basic open sets of $(S^1,d_2)$ and $(S^1,d)$ ($ d $ as defined in above Note \ref{n:metric}) are the arcs of $S^1$ without end points. Consequently, the metrics `$d_2$' and `$d$' produce the same topology $\ta$ on $S^1$. Further, $(S^1,d)$ is a second countable metric space because $\R^2$ with usual topology is second countable. Also any open ball in $ (S^1,d) $ being an arc is connected.
\label{n:2ndcount}\end{Not}

\begin{Th}
	$(S^1,d)$ is a tms with respect to suitably defined $ \sigma $-algebra and measure.
\label{t:S1tms}\end{Th}

\begin{proof}
	For the metric space $(S^1,d)$ (explained in Note \ref{n:metric}), we define an outer measure $ \nu $ as in Theorem \ref{t:outer} with the help of the metric $ d $. Next we define $\mathfrak M$ as the collection of all $\nu$-measurable sets of $S^1$ and $\mu:=\nu|_{\mathfrak M}:\mathfrak M\rightarrow[0,\infty]$. Then by Theorem \ref{t:car}, $ \mathfrak M $ is a $ \sigma $-algebra on $ S^1 $ and $ \mu $ is a measure on $\mathfrak M$. Also by Remark \ref{rm:complmsp},  $(S^1,\mathfrak M,\mu)$ is a complete measure space and $\mathfrak M$ contains the Borel $\sigma$-algebra $\mathscr{B}$ on $S^1$. 
	
	Now by Note \ref{n:2ndcount}, $ (S^1,d) $ is a second countable metric space where each open ball is connected. Then in view of Theorem \ref{t:mtms} we can say that $(S^1,d,\mathfrak M,\mu)$ is a tms.
\end{proof}

\begin{Rem} Looking at the proof of the Theorem \ref{t:S1tms} and Note \ref{n:2ndcount} we can say that $ (S^1,d_2,\mathfrak M,\mu) $ is also a tms, where $ d_2 $ is the Euclidean metric defined on $ S^1 $ and $ \mathfrak M,\mu $ are defined as in Theorem \ref{t:S1tms} with the help of $ d_2 $ instead of $ d $.\label{rm:S1tms}
\end{Rem}

\begin{Not}
	From Remark \ref{rm:S1tms} , we can say that every rectifiable curve of $\R^2$ is a tms, where the $ \sigma $-algebra and measure are constructed as in Theorems \ref{t:car} and \ref{t:outer} with the help of the Euclidean metric $ d_2 $.  
\label{n:rectitms}\end{Not}

\section{Absolutely continuous function on tms}

In this section, we shall introduce the concept of absolutely continuous function on tms. Some properties of absolutely continuous functions are discussed. Finally several examples of absolutely continuous functions have been presented. In Example \ref{e:ac} (vi), we have unveiled how Lebesgue measure fails to make projection maps and certain trigonometric functions absolutely continuous functions on $\R^2$ despite of existence of natural sense of absolute continuity. This situation is alleviated with the application of measure constructed in Theorem \ref{t:mtms} on $\R^2$. 

\begin{Nota} For a tms $(X,\ta,\M,m)$, for each $\delta>0,$ let us define $\pt{}{\delta}$ to be the collection of all countable family of disjoint open  connected sets $\{E_{i}\}_{i}$ with $ \displaystyle\sum_{i}^{}m(E_{i})<\delta$.
\end{Nota}

\begin{Def}
	Let $(X,\ta,\M,m)$ be a tms. A function $f:X\rightarrow\K$ ($\K$ is the field of real or complex numbers) is said to be an \textit{absolutely continuous function} if for every $\varepsilon>0$, $\exists\,\delta>0$ such that for any $\{D_{i}\}_{i}\in\pt{}{\delta}$ we have $\displaystyle\sum_{i}\omega(f,D_{i})<\varepsilon$  where\\ $\omega(f,D_{i}):= \sup\big\{| f(x)-f(y)| : x,y\in D_{i}\big\}$ for all $i$.
	\label{d:ac}
\end{Def}

\begin{Th}
	Every absolutely continuous function is uniformly continuous.\label{t:absunif}
\end{Th}
\begin{proof}
	Let $(X,\ta,\M,m)$ be a tms and $f:X\rightarrow\K$ be an absolutely continuous function. Let $\varepsilon>0$ be arbitrary. Then $\exists\,\delta>0$ such that for any open connected set $V$ with $m(V)<\delta,$  $\omega(f,V)<\varepsilon$. Hence sup $\{| f(x)-f(y)|: x,y \in V\}<\varepsilon$ whenever $V$ is an open connected set with measure $<\delta\implies| f(x)-f(y)|<\varepsilon$ for all $(x,y)\in N(\delta)$ where $\{N(\eta):\eta>0\}$ is the base for uniformity $\mathscr{U}$ on tms $X$ (as explained in Theorem \ref{t:unif}). Therefore $f:X\rightarrow\K$ is an uniformly continuous function.
\end{proof}

\begin{Rem}
	Every absolutely continuous function being uniformly continuous is continuous and therefore is measurable. Also if $ f $ is absolutely continuous on a tms $ X $  then for any constant $ c $, the function $ f+c $ is also absolutely continuous.
	\label{rm:cont}\end{Rem}

\begin{Th}
	For the tms $\big([a,b],\ta_{ab},\mathscr{L}([a,b]),\lambda\big)$, a function $f:[a,b]\rightarrow\K$ is absolutely continuous with respect to standard definition if and only if it is absolutely continuous with respect to the Definition \ref{d:ac} .
\label{t:compareac}\end{Th}

\begin{proof}
	Let $f:[a,b]\rightarrow\K$ be an absolutely continuous function with respect to standard definition. Then for every $\varepsilon>0,$ there exists $\delta>0$ such that for every finite family of disjoint sub-intervals $\{(a_{j},b_{j})\}_{j}$ of $[a,b]$  with $\displaystyle\sum_{j}(b_{j}-a_{j})<\delta,$  $\displaystyle\sum_{j}| f(b_{j})-f(a_{j})|<\varepsilon\cdots$ (1) Let $\{E_{i}\}_{i}\in\pt{}{\delta}$ be arbitrary. Then $\displaystyle\sum_{i}\lambda(E_{i})<\delta$.\\
	\textbf{Case 1:} Let $\{E_{i}\}_{i=1}^n$ be a finite family. Thus for every $i=1,2,...,n$, $E_{i}=(a_{i},b_{i})$. For $ i=1,2,...,n $, let $x_i,y_i\in E_{i}$ be arbitrarily chosen such that $x_i< y_i$.
	Hence
		$| f(x_i)-f(y_i)|\leq| f(a_{i})-f(x_{i})|+| f(x_i)-f(y_i)|+| f(y_{i})-f(b_{i})|$ for all $i=1,2,...,n
		\implies\displaystyle\sum_{i=1}^{n}| f(x_{i})- f(y_{i})|\leq\sum_{i=1}^{n}\big\{| f(a_i)- f(x_i)|+| f(x_i)-f(y_i)|+| f(y_i)-f(b_i)|\big\}\cdots\cdots$ (2)\\
	Now $\{(a_i,x_i),(x_i,y_i),(y_i,b_i)\}_{i=1}^{n}$ is a finite family of disjoint sub-intervals of $[a,b]$ such that $\displaystyle\sum_{i=1}^{n}\big\{|x_i-a_i|+| x_i-y_i|+| y_i-b_i|\big\} =\sum_{i=1}^{n}(b_i-a_i)<\delta$. Therefore by (1), we have $\displaystyle\sum_{i=1}^{n}\big\{| f(x_i)-f(a_i)|+| f(x_i)-f(y_i)|+| f(b_i)-f(y_i)|\big\}<\varepsilon\implies\sum_{i=1}^{n}| f(x_i)-f(y_i)|<\varepsilon$ [by (2)]. Thus $\displaystyle\sum_{i=1}^{n}\sup \big\{| f(x_i)-f(y_i)|:x_i,y_i\in E_i\big\}= \sup \left\{\sum_{i=1}^{n}| f(x_i)-f(y_i)| : x_i,y_i\in E_{i}\right\}\leq\varepsilon$. Hence $\displaystyle\sum_{i=1}^{n}\omega(f,E_{i})\leq\varepsilon$.\\
	\textbf{Case 2:} Let $\{E_i\}_{i=1}^\infty$ be an infinite disjoint sequence of open connected subsets of $[a,b]$ such that $\displaystyle\sum_{i\in\mathbb{N}}^{}\lambda(E_{i})<\delta$. Choose $k\in \mathbb{N}$ and fix it. Then $\displaystyle\sum_{i=1}^{k}\lambda(E_i)<\delta\implies\sum_{i=1}^{k}\omega(f,E_{i})\leq\varepsilon$ (by Case 1). Therefore $\displaystyle\lim_{k\to\infty}\sum_{i=1}^{k}\omega(f,E_i)\leq\varepsilon\implies\sum_{i=1}^{\infty}\omega(f,E_i)\leq\varepsilon$. Thus $f$ is absolutely continuous with respect to Definition \ref{d:ac} .
	
	Conversely, let $f:[a,b]\rightarrow\K$ be absolutely continuous with respect to Definition \ref{d:ac} . Hence for every $\varepsilon>0$ there exists $\delta>0$ such that for every $\{E_{i}\}_{i}\in\pt{}{\delta},$  $\displaystyle\sum_{i}\omega(f,E_i)<\varepsilon$. Let $\{(a_i,b_i)\}_{i=1}^{n}$ be a disjoint family of sub-intervals of $[a,b]$ such that $\displaystyle\sum_{i=1}^{n}| a_i-b_i|<\delta$. Then clearly $\{(a_i,b_i)\}_{i=1}^{n}\in\pt{}{\delta}$. Hence by Definition \ref{d:ac} ,  $\displaystyle\sum_{i=1}^{n}\omega(f,(a_i,b_i))<\varepsilon\implies\sum_{i=1}^{n}| f(b_{i}) - f(a_{i})|\leq\sum_{i=1}^{n}\omega(f,(a_i,b_i))<\varepsilon$ \big[since $ f $ being absolutely continuous according to Definition \ref{d:ac}, is continuous on $ [a,b] $, by Remark \ref{rm:cont}\big]. Therefore $f:[a,b]\rightarrow\K$ is absolutely continuous with respect to standard definition of absolute continuity.
\end{proof}

\begin{Th}
	If $ f $ is absolutely continuous on a tms $ X $, then it is absolutely continuous on every open subtms of $ X $.\label{t:openabs}
\end{Th}

\begin{proof}
	Let $(X,\ta,\M,m)$ be a tms and $ Y $ be an open subtms of $ X $. Also let $ f:X\to\K $ be absolutely continuous.	Let $ \delta>0 $ corresponds to arbitrary $ \varepsilon>0 $ in the definition of absolute continuity of $ f $ on $ X $. Let $ \{E_i\}_i\in\mathcal P_{\delta} $ be arbitrary, where $ E_i $'s are disjoint open connected subsets of $ Y $. Then by Lemma \ref{l:con}, $ E_i $'s are connected sets in $ X $ also. Moreover, $ Y $ being open, $ E_i $'s are open in $ X $. Therefore, by absolute continuity of $ f $ on $ X $ we have $ \displaystyle\sum_i\omega(f,E_i)<\varepsilon $. Thus $ f $ is absolutely continuous on $ Y $.
\end{proof}

\begin{Th}
	Let $(X,\ta,\M,m)$ be a tms and $AC(X)$ be the collection of all $ \K $-valued absolutely continuous functions on $X$. Let $f,g\in AC(X)$. Then\\ (i) $f+g \in AC(X)$,\\ (ii) $\alpha\cdot f\in AC(X)$ for all $\alpha\in\K$,\\ (iii) $f\cdot g\in AC(X)$ provided $f$ and $g$ both are bounded functions,\\ (iv) $\frac{1}{f}\in AC(X)$, if $| f|\geq K$ on $ X $ for some $ K>0 $,\\ (v) $| f|\in AC(X)$. 
	\label{t:ac}\end{Th}

\begin{proof}
	(i) Since $f,g\in AC(X),$ for any $\varepsilon>0,\exists\,\delta>0$ such that for every $\{E_i\}_i\in\pt{}{\delta}$, $\displaystyle\sum_{i}\omega(f,E_i)<\tfrac{\varepsilon}{2}$ and $\displaystyle\sum_{i}\omega(g,E_i)<\tfrac{\varepsilon}{2}$ ...(1). Now for all $i\in\N$ and for all $x,y\in E_i$, $| (f+g)(x)-(f+g)(y)|=|\{f(x)-f(y)\}+\{g(x)-g(y)\}|\leq| f(x)-f(y)|+| g(x)-g(y)|\implies\sup \{|(f+g)(x)-(f+g)(y)|:x,y\in E_i\}\leq \sup \{| f(x)-f(y)|+| g(x)-g(y)|:x,y\in E_i\}= \sup \{| f(x)-f(y)|:x,y\in E_i\}+ \sup \{| g(x)-g(y)|: x,y\in E_i \}\implies\omega(f+g,E_i)\leq\omega(f,E_i)+\omega(g,E_i),\forall\, i\in\N\implies\displaystyle\sum_{i}\omega(f+g,E_i)\leq\sum_{i}\{\omega(f,E_i)+\omega(g,E_i)\}=\sum_{i}\omega(f,E_i)+\sum_{i}\omega(g,E_i)<\tfrac{\varepsilon}{2}+\tfrac{\varepsilon}{2}=\varepsilon$ (by (1)). Therefore $f+g\in AC(X)$.
	
	(ii) For all $x,y\in A\subseteq X$ and $\alpha\in\K,|(\alpha\cdot f)(x)-(\alpha\cdot f)(y)|=|\alpha|\cdot| f(x)-f(y)|$. Since $f\in AC(X)$, for every $\varepsilon>0,\exists\,\delta>0$ such that for every $\{E_{i}\}_{i}\in\pt{}{\delta}$, $\displaystyle\sum_{i}\omega(f,E_i)<\tfrac{\varepsilon}{(|\alpha|+1)}$ ...(2). Hence for every $ i\in\N$ and for every $x,y\in E_i,|(\alpha\cdot f)(x)-(\alpha\cdot f)(y)|=|\alpha|\cdot| f(x)-f(y)|\implies \sup \{|(\alpha\cdot f)(x)-(\alpha\cdot f)(y)|:x,y\in E_i\}=|\alpha|\cdot \sup \{| f(x)-f(y)|:x,y\in E_i\}\implies\omega(\alpha\cdot f,E_i)=|\alpha|\cdot\omega(f,E_i),\forall\, i\in \N\implies\displaystyle\sum_{i}\omega(\alpha\cdot f, E_i)=\sum_{i}|\alpha|\cdot\omega(f,E_i)=|\alpha|\cdot\sum_{i}\omega(f,E_i)<|\alpha|\cdot\tfrac{\varepsilon}{1+|\alpha|}<\varepsilon$ (by (2)). Hence $\alpha\cdot f\in AC(X)$, for all $\alpha\in\K$.
	
	(iii) Since $f,g:X\rightarrow\K$ both are bounded functions, $\exists\, M>0$ such that $| f(x)|\leq M$, $| g(x)|\leq M$, for all $x\in X$. Now $f,g\in AC(X)\implies$ for every $\varepsilon>0,\exists\,\delta>0$ such that for every $\{E_i\}_i\in\pt{}{\delta},\displaystyle\sum_{i}\omega(f,E_i)<\tfrac{\varepsilon}{2M}$ and $\displaystyle\sum_{i}\omega(g,E_i)<\tfrac{\varepsilon}{2M}$ ...(3).\\ Now for all $i\in\N$ and for all $x,y\in E_i,|(f\cdot g)(x)-(f\cdot g)(y)|=| f(x) g(x)-f(y)g(y)|=| f(x) g(x)-f(x) g(y)+f(x) g(y)-f(y) g(y)|\leq| f(x)|\cdot| g(x)-g(y) |+| g(y)|\cdot| f(x)-f(y) |\leq M\cdot| g(x)-g(y)|+ M\cdot| f(x)-f(y)|$. Hence $\sup \{|(f.g)(x)-(f.g)(y)|:x,y\in E_i\}$  $\leq M\cdot \sup \{| g(x)-g(y)|
	:x,y\in E_i\}\ + M\cdot \sup \{| f(x)-f(y)|:x,y\in E_i\}$, for all $i\in\N$  $\implies\omega(f\cdot g,E_i)\leq M\cdot\omega(g,E_i) + M\cdot\omega(f,E_i)$, for all $i\in \N\implies\displaystyle\sum_{i}\omega(f\cdot g,E_i)\leq M\cdot\sum_{i}\omega(g,E_i)+M\cdot\sum_{i}\omega(f,E_i)< M\cdot\tfrac{\varepsilon}{2M} + M\cdot\tfrac{\varepsilon}{2M} = \varepsilon$ (by (3)) $\implies f\cdot g \in AC(X)$.
	
	(iv) $| f|\geq K>0$ on $X \implies$ $f(x)\not=0$ and $|\frac{1}{f(x)}|\leq\frac{1}{K}$ for all $x\in X$. Since $f\in AC(X),$ for every $\varepsilon>0$ there exists $\delta>0$ such that for every $\{E_i\}_{i}\in\pt{}{\delta}$,  $\displaystyle\sum_{i}\omega(f,E_i)<\varepsilon\cdot K^2 ...(4)$. Now for all $i\in\N$ and for all $x,y\in E_i,\left|\frac{1}{f}(x)-\frac{1}{f}(y)\right|=\frac{| f(y)-f(x)|}{| f(x)|\cdot| f(y)|}\leq\frac{1}{K^2}| f(x)-f(y)|\implies \sup \left\{\left|\frac{1}{f}(x)-\frac{1}{f}(y)\right|:x,y\in E_i\right\}\leq \frac{1}{K^2}\cdot \sup \{| f(x)-f(y)|:x,y\in E_i\}\implies\omega(\frac{1}{f},E_i)\leq\tfrac{1}{K^2}\cdot\omega(f,E_i)$ for all $i\in\N\implies\displaystyle\sum_{i}\omega(\tfrac{1}{f},E_i)\leq\tfrac{1}{K^2}\cdot\sum_{i}\omega(f,E_i)<\tfrac{1}{K^2}\cdot\varepsilon\cdot K^2=\varepsilon$ (by (4)). Hence $\frac{1}{f}\in AC(X)$.
	
	(v) Since $f\in AC(X)$, for every $\varepsilon>0$ there exists $\delta>0$ such that for every $\{E_i\}_i\in\pt{}{\delta},$  $\displaystyle\sum_{i}\omega(f,E_i)<\varepsilon$. Now for all  $i\in\N$ and for all $x,y\in E_i$, $\big|| f|(x)-| f|(y)\big| = \big|| f(x)|-| f(y)|\big|\leq| f(x)-f(y)|\implies \sup \Big\{\big|| f|(x)-| f|(y)\big|:x,y\in E_i\Big\}\leq \sup \{| f(x)-f(y)|:x,y\in E_i\}$  $\implies\omega(| f|,E_i)\leq\omega(f,E_i)$ for all $i\in\N\implies\displaystyle\sum_{i}\omega(| f|,E_i)\leq\sum_{i}\omega(f,E_i)<\varepsilon\implies$  $| f|\in AC(X)$.
\end{proof}

\begin{Not}
	If $f,g\in AC(X),| g|\geq K>0$ on $X$ and $f$ is bounded on $X$, then $\frac{f}{g}\in AC(X)$ (follows from (iii) and (iv) of Theorem \ref{t:ac}).
\end{Not}

\begin{Cor}
	$ (AC(X),+,.) $ forms a vector space over $\K$. If $AC^b(X)$ denotes the collection of all bounded absolutely continuous functions on $X$, then $ AC^b(X) $ is a ring and algebra over $ \K $.
\end{Cor}

\begin{Def}
	Let $(X,\ta,\M,m)$ be a tms. Then $ m $ is said to be \textit{C-outer regular} if for every connected $ \M $-measurable set $ C $ with $ m(C)=0 $ and any $ \varepsilon>0 $, there is an open connected set $ U\supseteq C $ such that $ m(U)<\varepsilon $.
\end{Def}

\begin{Th}
	Let $(X,\ta,\M,m)$ be a tms, where $ m $ is C-outer regular. Let $f:X\rightarrow \K$ be an absolutely continuous function. If $ E $ is a connected $ \M $-measurable subset of $X$ of measure 0, then $\lambda(f(E))=0$ where $\lambda$ is the Lebesgue measure on $\K$.
\label{t:con}\end{Th}

\begin{proof}
	Since $f\in AC(X),$ for every $\varepsilon>0,\ \exists\,\delta>0$ such that for every $\{A_i\}_i\in\pt{}{\delta},$  $\displaystyle\sum_{i}\omega(f,A_i)<\varepsilon$. $E$ is a connected subset of $X$ with $m(E)=0$. So $ m $ being C-outer regular, there is an open connected set $ U\supseteq E $ such that $ m(U)<\delta $. Therefore, by definition of absolute continuity of $f,\ \omega(f,U)<\varepsilon\implies \sup \{| f(x)-f(y)|: x,y\in U\}<\varepsilon$  $\implies| f(x)-f(y)|<\varepsilon$ for all $x,y\in E\implies| f(x)-f(a)|<\varepsilon$ for all $x\in E$ and for some $a\in E\implies f(x)\in B(f(a),\varepsilon)$ for all $x\in E$. Therefore $f(E)\subseteq B(f(a),\varepsilon)\implies\lambda^*(f(E))\leq\lambda^*(B(f(a),\varepsilon))=\lambda(B(f(a),\varepsilon))$ \big[$\lambda^*$ is the Lebesgue outer measure on the power set $\mathscr{P}(\K)$ and $\lambda$ is the Lebesgue measure on Lebesgue $\sigma$-algebra $\mathscr{L}(\K)$ of $\K$\big]. Now if $ \K=\R $ then $ \lambda(B(f(a),\varepsilon))=2\varepsilon $ and if $ \K=\C $ then $ \lambda(B(f(a),\varepsilon))=\pi\varepsilon^2 $.	
	Since $\varepsilon>0$ is arbitrary, we have $\lambda^*(f(E))=0$ and consequently, $\lambda(f(E))=0$ [since $(\K,\mathscr{L}(\K),\lambda)$ is a complete measure space].
\end{proof}

\begin{Cor}
	Let $(X,\ta,\M,m)$ be a tms ($ m $ being C-outer regular) and $f$ be  a real-valued absolutely continuous function on $ X $. Then for any connected $ \M $-measurable set $ E $ in $ X $ with $ m(E)=0 $, $ f $ is constant on $ E $.
\label{c:cons}\end{Cor}

\begin{proof}
	$ f $ being absolutely continuous is continuous (by Remark \ref{rm:cont}) and hence $ f(E) $ is a connected set in $ \R $. So from Theorem \ref{t:con} it follows that $ f(E) $ is a singleton set in $ \R $.
\end{proof}

\begin{Ex}
	(i) For the tms $(\R,\ta_u,\mathscr{L}(\R),\lambda)$, the function $f:\R\rightarrow \R$ defined by $f(x):=x,\, x\in\R$ is an absolutely continuous function. In fact, for any open connected set $A$, $\omega(f,A)= \sup \{| f(x)-f(y)|:x,y\in A\}= \sup \{| x-y|:x,y\in A\}=\lambda(A)$. Let $\varepsilon>0$ be arbitrary and $\{A_i\}_i\in\pt{}{\varepsilon
	}$. Therefore $\displaystyle\sum_{i}\omega(f,A_i)=\sum_{i}\lambda(A_i)<\varepsilon$. Consequently, $f$ is absolutely continuous.

Clearly, every constant function on $ \R $ is absolutely continuous.

For any bounded above interval $ I $ in $ \R $, any polynomial over $ I $ with coefficients from $ \K $ is absolutely continuous, by Theorem \ref{t:ac} (i), (ii) and (iii).

(ii) For the tms $(\R,\ta_u,\mathscr{L}(\R),\lambda)$, the function $f:\R\rightarrow\R$ defined by $f(x):=\sin x,$ $\forall\,x\in\R$ is absolutely continuous. To justify this, let $A$ be an open connected subset of $\R$ of arbitrarily small length. Then $| f(x)-f(y)|=|\sin x-\sin y|=2|\cos(\frac{x+y}{2})\sin(\frac{x-y}{2})|\leq 2|\frac{x-y}{2}|=| x-y|$, for all $x,y\in A$. Thus $\omega(f,A)=\sup\{| f(x)-f(y)|:x,y\in A\}\leq \sup \{| x-y|: x,y\in A\}=\lambda(A)$. Let $\varepsilon>0$ be arbitrary and $\{A_i\}_i\in\pt{}{\varepsilon}$. Then $\displaystyle\sum_{i}\omega(f,A_i)\leq\sum_{i}\lambda(A_i)<\varepsilon$. Hence $f$ is absolutely continuous.
	
	Similarly, $f(x):=\cos x,\forall\, x\in\R$ is absolutely continuous.
	
	(iii) For the tms $([a,b],\ta_{ab},\mathscr{L}([a,b]),\lambda)$, every absolutely continuous function with respect to standard definition of absolutely continuity is also absolutely continuous with respect to Definition \ref{d:ac} (since by Theorem \ref{t:compareac} both the definitions are equivalent).
	
	(iv) For the tms $(S^1,d,\mathfrak M,\mu)$ [By Theorem \ref{t:S1tms}], the function $f:S^1\rightarrow\C$ defined by $f(z):=z,\, z\in S^1$ is an absolutely continuous function. In fact,
	if $C$ is any open connected subset of $S^1$ then $\omega(f,C)=\sup \{| f(z_1)-f(z_2)|: z_1,z_2\in C \} = \sup \{| z_1-z_2|:z_1,z_2\in C\}\leq\sup \{d( z_1,z_2):z_1,z_2\in C\}=\text{diam}(C)=\mu(C)$ [since with respect to the metric $ d $, arcs on $ S^1 $ without endpoints are the only open connected sets]. Let $\varepsilon>0$ be arbitrary and $\{C_i\}_i\in\pt{}{\varepsilon}$ be arbitrary. Then $\displaystyle\sum_{i}\omega(f,C_i)\leq\sum_{i}\mu(C_i)<\varepsilon$. Therefore $f$ is absolutely continuous. [Here $ S^1 $ is written as $ \{z\in\C:|z|=1\} $].
	
	(v) In the above example (iv), if we replace $ S^1 $ by the complex plane $ \C $ equipped with the usual metric $ d_2 $ and $ \mathscr{A},\mu $ are constructed using $ d_2 $, as explained in Theorem
	\ref{t:mtms}, then $ (\C,d_2,\mathscr{A},\mu) $ is a tms. Now the function $f(z):=z,\, z\in \C$ is an absolutely continuous function. In fact,
	if $C$ is any open connected subset of $\C$ then, $\omega(f,C)=\sup \{| f(z_1)-f(z_2)|: z_1,z_2\in C \} = \sup \{| z_1-z_2|:z_1,z_2\in C\}=\text{diam}(C)\leq\mu(C)$ [by construction of $\mu$ as in theorem \ref{t:outer}]. Then by similar argument as in (iv), our claim is justified.
	
	(vi) For the tms $(\R^2,\ta_u,\mathscr{L}(\R^2),\lambda)$, the function $f:\R^2\rightarrow\R$ defined by $f(x,y):=x,$ $\forall\,(x,y)\in\R^2$ is not absolutely continuous. In fact, if it is absolutely continuous then by Corollary \ref{c:cons}, $ f $ should be constant on every connected subset of $ \R^2 $ with measure 0. Here any straight line in $ \R^2 $ is a connected set having Lebesgue measure 0. Clearly $ f $ is non-constant on any non-vertical line. Thus projection maps are not absolutely continuous, if the underlying measure is the Lebesgue measure. Here it is easy to verify that Lebesgue measure is C-outer regular.
	
	By the similar argument we can say that the function $f(x,y):=\sin xy,$ $\forall\,(x,y)\in\R^2$ is not absolutely continuous.
	
	(vii) If we consider the tms $ (\R^2,d_2,\mathscr{A},\mu) $, $ d_2 $ being the usual metric on $ \R^2 $ and $\mathscr{A},\mu$ being as in Theorem \ref{t:mtms}, then $f(x,y):=x,$ $\forall\,(x,y)\in\R^2$ is absolutely continuous. This follows since for any open connected set $ C $ in $ \R^2 $, $ \omega(f,C)=\sup\{|f(x,y)-f(x',y')|:(x,y),(x',y')\in C\}=\sup\{|x-x'|:(x,y),(x',y')\in C\}\leq\text{diam}(C)\leq\mu(C) $.
	
	(viii) For the tms $(\R,\ta_u,\mathscr{L}(\R),\lambda)$, the function $f:\R\rightarrow\R$ defined by $f(x):=x^2$, $x\in\R$ is not uniformly continuous and hence is not absolutely continuous.
\label{e:ac}\end{Ex}
\begin{Not}
	Example \ref{e:ac} (vi) shows that Lebesgue measure fails to make projection maps and sine function absolutely continuous on $\R^n (n\geq 2) $ despite of preserving the intrinsic sense of absolute continuity. The main reason for such situation is the existence of zero measurable connected arcs in Lebesgue $\sigma$-algebra on $\R^n ( n\geq 2 ) $. It excludes a rich subclass of continuous functions on $\R^n ( n\geq 2 ) $ from studying their absolute continuity. This situation can be improved if the measure and $\sigma$-algebra as mentioned in Theorem \ref{t:mtms} is considered. A discussion on absolute continuity of projection map on $\R^2$ with that referred measure and $\sigma$-algebra is presented in Example \ref{e:ac} (vii).\label{n:lebesguefail}
\end{Not}

\begin{Ex}
	For the tms $(\R,\ta_u,\mathscr{L}(\R),\lambda)$, we define a function \\  \centerline{$ f(x):= \begin{cases} 			
		x\sin(\frac{1}{x}), & \text{if } x\neq 0\\
		0 ,& \text{if } x=0
		\end{cases} $}
	 First, we prove that $f$ is not absolutely continuous on tms $(0,1)$ with respect to Definition \ref{d:ac}. On the contrary, let $f$ be absolutely continuous on $(0,1)$ with respect to Definition \ref{d:ac}. Then for any $\varepsilon>0$ there exists a $\delta>0$ such that for any $\{E_i\}_i\in\pt{}{\delta}$, we have $\displaystyle\sum_{i=1}^\infty\omega(f,E_i)<\varepsilon.$ Since $E_i\subseteq(0,1)$ for all $i\in\N$ and $E_i$'s are open connected, $E_i$'s are of the form $(a_i,b_i)$ where $0\leq a_i<b_i\leq 1$ for all $i\in\N$. Hence $\displaystyle\sum_{i=1}^\infty|b_i-a_i|<\delta\implies\sum_{i=1}^n|b_i-a_i|<\delta$, $\forall\, n\in\N$ and this implies $\displaystyle\sum_{i=1}^\infty\omega(f,(a_i,b_i))<\varepsilon\implies\sum_{i=1}^n\omega(f,(a_i,b_i))<\varepsilon,\forall n\in\N ...(1)$. Since $f$ is continuous on $[0,1]$, $|f(a_i)-f(b_i)|\leq\omega(f,(a_i,b_i))$ for all $i\in\N$. Therefore $\displaystyle\sum_{i=1}^{n}|a_i-b_i|<\delta\implies\sum_{i=1}^{n}|f(a_i)-f(b_i)|\leq\sum_{i=1}^{n}\omega(f,(a_i,b_i))<\varepsilon$ for all $n\in\N$ (from $(1)$). Since $\{(a_i,b_i)\}_{i=1}^n$ is any finite family of open subintervals of [0,1] for every $n\in\N$, we can conclude that $f$ is absolutely continuous on $[0,1]$ with respect to standard definition ------ which is a contradiction. Hence our assumption is wrong i.e. $f$ is not absolutely continuous on $(0,1)$ with respect to Definition \ref{d:ac}. 
		
		Now, suppose $f$ is absolutely continuous on $\R$ with respect to Definition \ref{d:ac}. For $(0,1)$ being an open subtms of $\R$, $f$ is absolutely continuous on $(0,1)$ with respect to Definition \ref{d:ac}, by Theorem \ref{t:openabs} ------ which is a contradiction by above discussion. Therefore $f$ is not absolutely continuous on tms $\R$.\qed\label{e:notabs} 
\end{Ex}
\begin{Not}
	The function studied in Example \ref{e:notabs} is of immense importance. This is because the referred function is uniformly continuous on $\R$ but not absolutely continuous on tms $\R$ with respect to Definition \ref{d:ac}. This says that converse of Theorem \ref{t:absunif} does not hold in general.
\end{Not}

\begin{Ex}
	Consider the tms $(\R,\ta_u,\mathscr{L}(\R),\lambda)$ and let $f\in L^1(\R)$. Define\\ $F(x):=\displaystyle\int_{[-| x|,| x|]}f(t)dt, \forall\,x\in\R$ [Here `$ dt $' denotes element of Lebesgue measure $ \lambda $]. Then $F$ is an absolutely continuous function on $\R$.\label{e:int}

\noindent{\bf Justification :} Since $f\in L^1(\R)$, $F$ is well-defined. For each $n\in\N$, define\\  \centerline{$ \phi_n(t):= \begin{cases} 
|
f(t)|, & \text{if } | f(t)|\leq n\\
n ,& \text{otherwise }
 \end{cases} $}
Then each $\phi_n$ is a measurable function such that $\displaystyle\lim_{n\to\infty}\phi_n(x)=|f(x)|$ for all $x\in\R$. Also $\phi_n\leq| f|$ for all $n\in\N$. Since $f\in L^1(\R),$ by dominated convergence theorem (DCT), $\displaystyle\lim_{n\to\infty}\int_{\R}\big|| f|-\phi_n\big| dt=0$. Hence for every $\varepsilon>0,\exists\, p\in\N$ such that $\displaystyle\int_{\R}\big|| f| - \phi_n\big| dt<\tfrac{\varepsilon}{4}$ for all $ n\geq p$ ...(i)\\ Choose $\delta=\frac{\varepsilon}{4\cdot p}>0$ and let $\{E_i\}_i\in\pt{}{\delta}$. Now for all $i\in\N$ and for any $x,y\in E_i$ with $| x|<| y|,$ we have $| F(x)-F(y)|=\left|\displaystyle\int_{[-| x|,| x|]} f dt-\int_{[-| y|,| y|]} fdt \right|\leq \displaystyle\int_{I_1\cup I_2}| f|dt$, where $I_1 = [-| y|,-| x|),I_2 = (| x|,| y|]$. Also $\displaystyle\int_{I_1\cup I_2}| f| dt\leq\int_{E_i\cup(-E_i)}| f|dt\leq\int_{E_i}| f|dt +\int_{(-E_i)}| f|dt$. Hence $| F(x)-F(y)|\leq\displaystyle\int_{E_i}| f|dt + \int_{(-E_i)}| f|dt$ for any choice of $x,y\in E_i$, $i\in\N\implies\omega(F,E_i)\leq\displaystyle\int_{E_i}| f|dt+\int_{(-E_i)}| f|dt$  for all $i\in\N$ ...(ii)\\
Now $\displaystyle\sum_{i}\int_{E_i}| f|dt + \sum_{i}\int_{(-E_i)}| f|dt\leq\sum_{i}\int_{E_i}\big|| f|-\phi_{p}\big|dt$ + $\displaystyle\sum_{i}\int_{(-E_i)}\big|| f|-\phi_{p}\big|dt$ + $\displaystyle\sum_{i}\int_{E_i}| \phi_{p}|dt$ + $\displaystyle\sum_{i}\int_{(-E_i)}| \phi_{p}|dt $  $ \leq\displaystyle\int_{\cup E_i}\big|| f|-\phi_{p}\big|dt + \int_{\cup(-E_i)}\big|| f|-\phi_{p}\big|dt+\sum_{i}p\cdot \lambda(E_i) + \sum_{i}p\cdot \lambda(-E_i)$ [By monotone convergence theorem (MCT)] $\leq 2\cdot\displaystyle \int_{\R}\big|| f|-\phi_{p}\big|dt + 2\cdot\sum_{i}p\cdot \lambda(E_i) $...(iii) [Since $ \lambda(E_i)=\lambda(-E_i),\forall\,i\in\N $].\\ Therefore from (i), (ii) and (iii), it follows that $\displaystyle\sum_{i}\omega(F,E_i)< 2\cdot\tfrac{\varepsilon}{4}+2p\delta$ = $\tfrac{\varepsilon}{2}+2\cdot p\cdot\frac{\varepsilon}{4p} = \varepsilon$. Consequently, $F$ is absolutely continuous.\qed
\end{Ex}

\begin{Ex}
	Consider the tms $(\R,\ta_u,\mathscr{L}(\R),\lambda)$ and let $f\in L^1(\R)$. Define\\ $G(x):=\displaystyle\int_{(-\infty,x]}^{} f(t)dt,\forall\, x\in\R$. Then $G$ is an absolutely continuous function on $ \R $. \\
\noindent\textbf{Justification :} Since $f\in L^1(\R)$, $G$ is well-defined. Now for each $n\in\N,\,\phi_n$ is defined as in Example \ref{e:int} . Then by DCT, $\displaystyle\lim_{n\to\infty}\int_{\R}\big|| f|-\phi_n\big|dt=0$. Therefore for every $\varepsilon>0,\exists\,p\in\N$ such that $\displaystyle\int_{\R}\big|| f|-\phi_n\big|dt<\tfrac{\varepsilon}{2}$ for all $n\geq p$ ...(i)\\ Choose $\delta=\frac{\varepsilon}{2p}>0$ and let $\{E_i\}_i\in\pt{}{\delta}$. Now for each $i\in\N$ and for any $x,y\in E_i$, $| G(x)-G(y)|=\displaystyle\left|\int_{(-\infty,x]}fdt-\int_{(-\infty,y]}fdt\right|\leq\int_{E_i}| f|dt\implies\omega(G,E_i)\leq\int_{E_i}| f|dt,\forall\, i\in\N\implies$  $\displaystyle\sum_{i\in\N}\omega(G,E_i)\leq\sum_{i\in\N}\int_{E_i}| f|dt\leq\sum_{i\in\N}\int_{E_i}\big|| f|-\phi_p\big|dt + \sum_{i\in\N}\int_{E_i}\phi_pdt\leq\int_{\cup E_i}\big|| f|-\phi_p\big|dt + \sum_{i\in\N}p.\lambda(E_i)\leq
\int_{\R}\big|| f|-\phi_{p}\big|dt + p\cdot\sum_{i\in\N}\lambda(E_i)<\tfrac{\varepsilon}{2} + p\cdot\tfrac{\varepsilon}{2p}=\varepsilon$ (from (i)). Therefore $G$ is absolutely continuous.\qed
\label{e:int2}\end{Ex} 

\begin{Ex}
	For the tms $([0,1],\ta_{u},\mathscr{L}([0,1]),\lambda)$, the function $F(x):=\sqrt{x},\;x\in[0,1]$ is an absolutely continuous function. This is a standard  problem according to standard definition. Since the standard definition of absolute continuity is equivalent to our definition on $ [0,1] $ (by Theorem \ref{t:compareac}), our claim is justified.
\label{e:rootx1}\end{Ex}

\begin{Ex}
	For the tms $(\R,\ta_u,\mathscr{L}(\R),\lambda)$, the function $F:\R\rightarrow\R$ defined by \centerline{$ 	F(x):=\begin{cases}
		\sqrt{x},&\text{if } x\in(a,\infty),a>0\\
		0,&\text{otherwise }
	\end{cases} $}
	is absolutely continuous on $(a,\infty)$. \\
\noindent\textbf{Justification :} Let $f:\R\rightarrow\R$ be a function defined by
$ 	f(t):=\begin{cases}
	\frac{1}{2\sqrt{t}},& \mbox{if } t\geq a\\
	0,&\mbox{otherwise }
\end{cases} $\\
Then $f$ is a non-negative monotonically decreasing bounded function on $(a,\infty)$. Let $M :=\sup \{| f(t)|:t\in(a,\infty)\}>0$. Clearly \begin{align*}
F(x)=\left\{\begin{array}{cc}\displaystyle
\int_{(-\infty,x]}f(t)dt + \sqrt{a},\;& \mbox{if }x> a\\
0\;,&\mbox{otherwise}\end{array}\right.
\end{align*}
Let $\varepsilon>0$ be arbitrary. Choose $\delta=\frac{\varepsilon}{M}$. Let $\{E_i\}_i\in\pt{}{\delta}$ such that $ \displaystyle\bigcup_{i=1}E_i\subseteq (a,\infty) $. Now $\forall \,i\in \N$ and for any $x,y\in E_i$ with $x< y$, we have\\ $| F(x)-F(y)| = \left|\displaystyle\left\{\int_{(-\infty,x]}f\;dt + \sqrt{a}\right\}-\left\{\int_{(-\infty,y]}f\;dt + \sqrt{a}\right\}\right| = \displaystyle\left| \int_{(x,y]}^{}f\;dt\right|\leq\int_{E_i}| f|\;dt$. Hence $\omega(F,E_i)\displaystyle\leq\int_{E_i}| f|\;dt,\forall\, i\in\N\implies\displaystyle\sum_{i\in\N}\omega(F,E_i)\leq\sum_{i\in\N}\int_{E_i}| f|\;dt\leq\sum_{i\in\N}M\cdot\lambda(E_i)< M\cdot\tfrac{\varepsilon}{M}=\varepsilon$. Hence $F$ is absolutely continuous on $(a,\infty)$ for any $ a>0 $.\qed
\label{e:rootx2}\end{Ex}

\begin{Ex}
	For the tms $([0,\infty),\ta_u,\mathscr{L}([0,\infty)),\lambda)$ , the function $F(x):=\sqrt{x}$, $\forall\,x\in[0,\infty)$ is absolutely continuous on $[0,\infty)$.\\
\noindent\textbf{Justification :}
	$F(x)=\sqrt{x},\; x\geq0$ is absolutely continuous separately on $[0,1]$ and on $(1,\infty)$ (by Examples \ref{e:rootx1} and \ref{e:rootx2}). Now we prove absolute continuity of $F$ on entire $[0,\infty)$. Let $\varepsilon>0$ be arbitrary. Then $\exists\,\delta>0$ such that for every $\{D_i\}_i\in\pt{}{\delta}$ with $ \displaystyle\bigcup_iD_i\subseteq[0,1] $ and for every $\{S_i\}_i\in\pt{}{\delta}$ with $ \displaystyle\bigcup_iS_i\subseteq(1,\infty) $, $\displaystyle\sum_{i}\omega(F,D_i)<\tfrac{\varepsilon}{2}$, $\displaystyle\sum_{i}\omega(F,S_i)<\tfrac{\varepsilon}{2}$.
	
	 Let $\{E_i\}_i\in\pt{}{\delta}$ be arbitrary such that $ \displaystyle\bigcup_iE_i\subseteq[0,\infty) $.\\	
	\textbf{Case 1:} There exists exactly one $k\in\N$ such that $E_k\cap[0,1]\not= \emptyset$ and $E_k\cap(1,\infty)\not=\emptyset$. Now $\omega(F,E_k)= \sup \{| F(x)-F(y)|:x,y\in E_k\} = \sup \{| F(x)-F(y)|: x,y\in E_k\cap [0,1]\} + \sup \{| F(x)-F(y)|: x,y\in E_k\cap(1,\infty)\}$ (since F is a monotone function with no point of discontinuity) ...(i).\\ Let $\varLambda_1:=\{i:E_i\subseteq[0,1]\},\varLambda_2:=\{i:E_i\subseteq(1,\infty)\}$. Then $\displaystyle\sum_{i\in\varLambda_1}\lambda(E_i)$ + $\displaystyle\lambda(E_k\cap[0,1])\leq\sum_{i}\lambda(E_i)<\delta\implies\sum_{i\in\varLambda_1}\omega(F,E_i) $ + $ \omega(F,E_k\cap[0,1])<\frac{\varepsilon}{2}$ ...(ii).\\ Also $\displaystyle\sum_{i\in\varLambda_2}\lambda(E_i)$ + $ \lambda(E_k\cap(1,\infty))\leq\displaystyle\sum_{i}\lambda(E_i)<\delta\implies\displaystyle\sum_{i\in\varLambda_2}\omega(F,E_i)$ + $ \omega(F,E_k\cap(1,\infty))<\frac{\varepsilon}{2}$ ...(iii). Therefore from (i),(ii) and (iii), we obtain $\displaystyle\sum_{i}\omega(F,E_i)<\varepsilon$.\\
	\textbf{Case 2:} No $E_i$ meets both $[0,1]$ and $(1,\infty)$. Let $\varLambda_1$ and $\varLambda_2$ be defined as in Case 1. Then $\displaystyle\sum_{i}\omega(F,E_i)=\sum_{i\in\varLambda_1}\omega(F,E_i) + \sum_{i\in\varLambda_2}\omega(F,E_i)<\tfrac{\varepsilon}{2} + \tfrac{\varepsilon}{2}=\varepsilon$.
		
	Hence $F$ is absolutely continuous on $[0,\infty)$.\qed
	\label{e:rootx3}\end{Ex}

From the previous discussions, we can easily further generalise the concept of absolutely continuous function as follows.

\begin{Def}
	Let $(X,\ta,\M,m)$ be a tms. A function $f:X\rightarrow\K^n\ (n\in\N,\K$ is the field of real or complex numbers) is said to be an \textit{absolutely continuous function} if for every $\varepsilon>0$, $\exists\,\delta>0$ such that for any $\{D_{i}\}_{i}\in\pt{}{\delta}$ we have $\displaystyle\sum_{i}\omega(f,D_{i})<\varepsilon$, where\\ $\omega(f,D_{i}):= \sup\big\{\| f(x)-f(y)\| : x,y\in D_{i}\big\}$ for all $i$.
	\label{d:knac}
\end{Def}

\begin{Def}
	Let $(X,\ta,\M,m)$ be a tms and $ N $ be normed linear space over $ \K $. A function $f:X\rightarrow N$ is said to be an \textit{absolutely continuous function} if for every $\varepsilon>0$, $\exists\,\delta>0$ such that for any $\{D_{i}\}_{i}\in\pt{}{\delta}$ we have $\displaystyle\sum_{i}\omega(f,D_{i})<\varepsilon$, where\\ $\omega(f,D_{i}):= \sup\big\{\| f(x)-f(y)\| : x,y\in D_{i}\big\}$ for all $i$.
	\label{d:nlsac}
\end{Def}

Looking at the proof of the Theorem \ref{t:ac} we can easily have the following analogous theorem.

\begin{Th}
	Let $(X,\ta,\M,m)$ be a tms and $ N $ be normed linear space over $ \K $. Then\\
	(i) the collection $ AC(X,\K^n)\ [n\in\N] $ of all absolutely continuous functions from $ X $ to $ \K^n $ is a vector space over $ \K $.\\
		(ii) the collection $ AC(X,N)$ of all absolutely continuous functions from $ X $ to $ N$ is a vector space over $ \K $.
\end{Th}

 In the next section we shall define another class of functions on tms which enriches the collection of absolutely continuous functions.

\section{Locally Lipschitz function on tms}

 Here we have introduced the concept of locally Lipschitz functions on tms. Its association with absolutely continuous functions has been analysed. Further, we have observed the equivalence of boundedness and absolute continuity of linear functionals on separable normed linear space which is a tms as explained in Remark \ref{rm:nlstms}. Also, if the co-domain of absolutely continuous functions is enlarged from $\K$ ($\K$ is the field of real or complex numbers) to any normed linear space, the notions pertinent to boundedness and absolute continuity produce same classes of linear transformations on any separable normed linear space with the participation of the measure and $\sigma$-algebra as discussed in Theorem \ref{t:mtms}. 
 
\begin{Def}
	Let $(X,\ta,\M,m)$ be a tms. A function $f:X\rightarrow\K$ is called a \emph{locally Lipschitz function} if $\exists\,L>0$ such that for every open connected set $E$, $| f(x)-f(y)|\leq L\cdot m(E)$, for all $x,y\in E$. 
\end{Def}

\begin{Rem}
	Every Lipschitz function on $\R$ is locally Lipschitz on $(\R,\ta_{u},\mathscr{L}(\R),\lambda)$.
\label{rm:Lip}\end{Rem}

\begin{Th}
	Let $(X,\ta,\M,m)$ be a tms and $f$ and $g$ be two locally Lipschitz functions on $ X $ and $\alpha\in\K$. Then\\ (i) $f+g,\ \alpha\cdot f$ are locally Lipschitz.\\ (ii) $f\cdot g$ is locally Lipschitz, provided $f,g$ both are bounded functions on $X$.\\ (iii) $\frac{1}{f}$ is locally Lipschitz provided $| f|\geq k$, for some $k>0 $.\\ (iv) $| f|$ is locally Lipschitz.
\end{Th}

\begin{proof}
	Straightforward.
\end{proof}

\begin{Cor}
	The collection of all locally Lipschitz functions on a tms is a vector space over $ \K $ and the collection of all bounded locally Lipschitz functions on a tms forms a commutative unitary ring as well as a commutative unital algebra over $ \K $.
\end{Cor}

\begin{Th}
	Let $(X,\ta,\M,m)$ be a tms. Then every locally Lipschitz function on $X$ is absolutely continuous.
\label{t:lac}\end{Th}

\begin{proof}
	Let $f:X\rightarrow\K$ be a locally Lipschitz function. Then $\exists\,L>0$ such that for every open connected set $E$, we have $| f(x)-f(y)|\leq L\cdot m(E)$, $\forall\,x,y\in E\implies$  $\omega(f,E)\leq L\cdot m(E)$  $\cdots\cdots$(1). Let $\varepsilon>0$ be arbitrary and $\{E_i\}_i\in\pt{}{\frac{\varepsilon}{L}}$ be arbitrary. Then $\displaystyle\sum_{i}m(E_i)<\tfrac{\varepsilon}{L}$  $\implies\displaystyle L\cdot \sum_{i} m(E_i)<\varepsilon$. Therefore from (1), $\displaystyle\sum_{i}\omega(f,E_i)\leq L\cdot\sum_{i}m(E_i)<\varepsilon$. Hence $f$ is an absolutely continuous function.
\end{proof}

\begin{Ex}
	(i) For the tms $(\R,\ta_u,\mathscr{L}(\R),\lambda),$ the functions $f(x):=x,$   $g(x):=\cos x,$   $h(x):=\sin x,\forall\, x\in\R$ all are locally Lipschitz functions (by Remark \ref{rm:Lip}).
	
	(ii) Consider the tms $([0,\infty),\ta_u,\mathscr{L}([0,\infty)),\lambda)$ and a function $f:[0,\infty)\rightarrow\K$ defined by $f(x):=\sqrt{x}$ for all $x\in[0,\infty)$. We prove that $f$ is not locally Lipschitz. On the  contrary, let us assume that $f$ is locally Lipschitz. Then $\exists\, M>0$ such that for every open connected set $E$, $| f(x)-f(y)|\leq M\cdot\lambda(E)$, for all $x,y\in E$ $ \implies\omega(f,E)\leq M\cdot\lambda(E) $. Let $E=\big(0,\frac{1}{(1+M)^2}\big)$. Then $ \omega(f,E)=\frac{1}{1+M} $ and $ \lambda(E)=\frac{1}{(1+M)^2}$  $\implies\frac{1}{1+M}\leq \frac{M}{(1+M)^2}\implies M+1\leq M$ ------ which is a contradiction. As a consequence, $f$ is not locally Lipschitz. \label{e:rt}\end{Ex}

\begin{Rem}
	Examples \ref{e:rootx3} and \ref{e:rt}(ii) justify that there exists absolutely continuous function which is not locally Lipschitz. 
\end{Rem}

\begin{Ex}
	Consider the separable normed linear space $ L^1(\R) $ which is a tms, by Remark \ref{rm:nlstms}. Here the $ \sigma $-algebra $ \mathscr{A}	 $ and the measure $ \mu $ are as described in Theorem \ref{t:mtms}. We define a map $ T:L^1(\R)\to\K $ by $ T(f):=\displaystyle\int_{\R}f(t)dt,\ \forall\,f\in L^1(\R) $. Since $ f\in L^1(\R) $, $ T $ is well-defined. Let $\mathcal{E} $ be an open connected set in $ L^1(\R) $. Then for any $ f,g\in \mathcal{E} $, $ |T(f)-T(g)|=\displaystyle\left|\int_{\R}f(t)dt-\int_{\R}g(t)dt\right|\leq\int_{\R}|f(t)-g(t)|dt=\|f-g\|_1\leq\text{diam}(\mathcal{E})\leq\mu(\mathcal{E}) $. Thus $ T $ is a locally Lipschitz function on $ L^1(\R) $ and hence is absolutely continuous on $ L^1(\R) $ [By Theorem \ref{t:lac}].
\end{Ex}

\begin{Ex}
	Consider the separable normed linear space $ L^p(\R^n)$ (where $n\in\N,p>1$) which is a tms, by Remark \ref{rm:nlstms}. Here the $ \sigma $-algebra $ \mathscr{A}$ and the measure $ \mu $ are as described in Theorem \ref{t:mtms}. For each $ h\,(\neq0)\in L^q(\R^n) $ \big(where $ q>1 $ such that $ \frac{1}{p}+\frac{1}{q}=1 $\big) we define a map $ T_h:L^p(\R^n)\to\K $ by $ T_h(f):=\displaystyle\int_{\R}f(t)h(t)dt,\ \forall\,f\in L^p(\R^n) $. Then $ \displaystyle\int_{\R}|f(t)h(t)|dt\leq\left(\int_{\R}|f(t)|^pdt\right)^{\frac{1}{p}}\cdot\left(\int_{\R}|h(t)|^qdt\right)^{\frac{1}{q}}=\|f\|_p\|h\|_q<\infty $ [By H\"{o}lder's inequality]. Thus $ T_h $ is well-defined. 
	
	Let $\mathcal{E} $ be an open connected set in $ L^p(\R^n) $. Then for any $ f,g\in \mathcal{E} $, $ |T_h(f)-T_h(g)|=\displaystyle\left|\int_{\R}f(t)h(t)dt-\int_{\R}g(t)h(t)dt\right|\leq\int_{\R}|f(t)-g(t)||h(t)|dt\leq\|f-g\|_p\|h\|_q\leq\|h\|_q\cdot\text{diam}(\mathcal{E})\leq\|h\|_q\cdot\mu(\mathcal{E}) $. Thus $ T_h $ is a locally Lipschitz function on $ L^p(\R^n) $ and hence is absolutely continuous on $ L^p(\R^n) $ [By Theorem \ref{t:lac}].\qed
\end{Ex}

We can generalise this example to some extent.

\begin{Ex}
	Let $ G $ be a second countable locally compact abelian Hausdorff topological group. Then $ L^p(G)\ (p>1) $ is a separable normed linear space and hence $ \big(L^p(G),\|.\|_p,\mathscr{A},\mu\big) $ is a tms, $ \mathscr{A},\mu $ being defined as in Theorem \ref{t:mtms}. Now for each $ h\,(\neq0)\in L^q(G) $ \big(where $ q>1 $ such that $ \frac{1}{p}+\frac{1}{q}=1 $\big), the map $ T_h(f):=\displaystyle\int_{G}f(t)h(t)dt,\ \forall\,f\in L^p(G) $ is locally Lipschitz and hence absolutely continuous on $ L^p(G) $ [Here `$ dt $' denotes the Haar measure on $ G $].
\end{Ex}

\begin{Th}
	Let $ X $ be a separable normed linear space. Then every bounded linear functional on $ X $ is Locally Lipschitz and hence absolutely continuous on $ X $.
\label{t:nls}\end{Th}

\begin{proof}
	By Remark \ref{rm:nlstms}, $ X $ is a tms, the $ \sigma $-algebra $ \mathscr{A}$ and the measure $ \mu $ are as described in Theorem \ref{t:mtms}. Let $ f $ be a nonzero bounded linear functional on $ X $. Let $ C $ be an open connected set in $ X $. Then for any $ x,y\in C $, $ |f(x)-f(y)|=|f(x-y)|\leq\|f\|.\|x-y\|\leq\|f\|.\text{diam}(C)\leq\|f\|.\mu(C) $. Thus $ f $ is a locally Lipschitz function (since $ \|f\|>0 $) and hence is absolutely continuous on $ X $. Clearly the zero function is locally Lipschitz as well as absolutely continuous.
\end{proof}

\begin{Cor}
	In a separable normed linear space $ X $, a linear functional $ f $ on $ X $ is absolutely continuous if and only if $ f $ is bounded.
\end{Cor}

\begin{proof}
	Since every absolutely continuous function is necessarily continuous [by Remark \ref{rm:cont}], this corollary follows from above Theorem \ref{t:nls}.
\end{proof}

\begin{Rem}
	If $ f $ is a bounded linear functional on a separable nomed linear space $ X $, then for any constant $ c $, the map $ f+c $ is absolutely continuous [by Remark \ref{rm:cont}].
\label{rm:bac}\end{Rem}

\begin{Ex}
	Consider the tms $ \big(\R^2,\|.\|_2,\mathscr{A},\mu\big) $ which is a normed linear space as well \big[Here $ \mu $ is constructed as in Theorem \ref{t:mtms} using the metric induced by the norm $ \|.\|_2 $\big]. By above Remark \ref{rm:bac}, for any real constants $ a,b,c $, the map $ f(x,y):=ax+by+c,\ \forall\,(x,y)\in\R^2 $ is absolutely continuous, since $ (x,y)\mapsto ax+by $ is a bounded linear functional on $ (\R^2,\|.\|_2) $.
	
	Here it is important to note that $ f $ will not be absolutely continuous if instead of the measure $ \mu $ the Lebesgue measure $ \lambda $ is considered, as justified in Example \ref{e:ac} (vi).\qed
\end{Ex}

\begin{Ex}
	Consider the separable Banach space $ \big(C[0,1],\|.\|_{\infty}\big) $ as tms, the $ \sigma $-algebra $ \mathscr{A} $ and measure $ \mu $ being defined as in Theorem \ref{t:mtms}. Define a map $ T:C[0,1]\to C[0,1] $ by $ (Tf)(x):=\displaystyle\int_0^xf(t)dt,\ \forall\,x\in[0,1],\forall\,f\in C[0,1] $. Since the integral is continuous on [0,1], $ T $ is well-defined. Clearly, $ T $ is linear.
	
	Let $ \mathcal E $ be an open connected set in $ C[0,1] $. Now for any $ f,g\in\mathcal E $, we have $ \|Tf-Tg\|_{\infty}=\|T(f-g)\|_{\infty}=\displaystyle\sup_{x\in[0,1]}\left|\int_0^x\big(f(t)-g(t)\big)dt\right|\leq\sup_{x\in[0,1]}\int_0^x\big|f(t)-g(t)\big|dt=\int_0^1\big|f(t)-g(t)\big|dt\leq\|f-g\|_{\infty}\leq\text{diam}(\mathcal E)\leq\mu(\mathcal E) $. 
	
	Let $ \varepsilon>0 $ be arbitrary and $ \{\mathcal E_i\}_i\in\mathcal P_{\varepsilon} $. Then\\ $ \displaystyle\sum_i\omega(T,\mathcal E_i)=\sum_i\sup\big\{\|Tf-Tg\|_{\infty}:f,g\in\mathcal E_i\big\}\leq\sum_i\mu(\mathcal E_i)<\varepsilon $.
	
	Thus $ T:C[0,1]\to C[0,1] $ is absolutely continuous according to the Definition \ref{d:nlsac}.
\end{Ex}

\begin{Th}
	Let $ X,N $ be two normed linear spaces over the field $ \K $ and $ X $ be separable. Then a linear transformation $ T:X\to N $ is absolutely continuous according to Definition \ref{d:nlsac} if and only if $ T $ is bounded.\label{t:abscontbddtran}
\end{Th}

\begin{proof}
	$ X $ can be considered as a tms, $ \mu $ being the measure as explained in Theorem \ref{t:mtms}. If $ T:X\to N $ is a bounded linear transformation then for any open connected set $ C $ in $ X $, we have $ \|Tx-Ty\|=\|T(x-y)\|\leq\|T\|.\|x-y\|\leq\|T\|.\text{diam}(C)\leq\|T\|.\mu(C)$, $\forall\,x,y\in C $. Thus $ \omega(T,C)=\sup\big\{\|Tx-Ty\|:x,y\in C\big\}\leq\|T\|.\mu(C) $.
	
	Let $ \varepsilon>0 $ be arbitrary and $ \delta=\frac{\varepsilon}{\|T\|+1} $. Choose any $ \{E_i\}_i\in\mathcal{P}_{\delta} $. Then $ \displaystyle\sum_i\omega(T, E_i)\leq\|T\|.\sum_i\mu(E_i)<\|T\|.\delta<\varepsilon $. Thus by Definition \ref{d:nlsac}, $ T $ is absolutely continuous.
	
	Converse is immediate, since every absolutely continuous function (even vector-valued) is necessarily continuous.
\end{proof}

\begin{Rem}
	If $ X,N $ are two normed linear spaces with $ X $ separable, then for any bounded linear transformation $ T:X\to N $ and any $ v\in N $ the map $ T+v $ is absolutely continuous according to the Definition \ref{d:nlsac}.
\end{Rem}

\begin{Th}
	On a separable normed linear space $X$, the binary operation addition is an absolutely continuous function on $ X\times X $ according to Definition \ref{d:nlsac}.
\end{Th}

\begin{proof}
	Since finite product of two separable topological spaces is separable, $X\times X$ is a separable normed linear space, where $\|(x,y)\|:=\text{max}\{\|x\|,\|y\|\}$ for all $x,y\in X.$ Now, we consider the $\sigma$-algebra $\mathscr{A}$ and measure $\mu$ on $X\times X$ as described in Theorem \ref{t:mtms}. Hence by Remark \ref{rm:nlstms}, $X\times X$ is a tms. Let us define the function $A:X\times X\to X$ by $A(x,y):=x+y$, for all $(x,y)\in X\times X$. Then obviously $A$ is a linear transformation. Now $\|A(x,y)\|=\|x+y\|\leq\|x\|+\|y\|\leq 2\cdot\max\{\|x\|,\|y\|\}=2\|(x,y)\|$, for all $(x,y)\in X\times X$. Hence $A$ is a bounded linear transformation on $X\times X$. Consequently, by Theorem \ref{t:abscontbddtran} we can conclude that $A$ is an absolutely continuous function on $X\times X$.
\end{proof}

\begin{Th}
	On a separable normed linear space $X$, for each $\alpha\in\K$ define a function $\zeta_\alpha(x):=\alpha\cdot x$, for all $x\in X$. Then $\zeta_\alpha$ is absolutely continuous on $X$, for all $\alpha\in\K.$ 
\end{Th}

\begin{proof} This theorem follows from Theorem \ref{t:abscontbddtran} since $\zeta_\alpha$ is a bounded linear transformation on $X$, for each $\alpha\in\K$.
\end{proof}

\begin{Th}
	On a separable normed linear space $X$, the norm function is absolutely continuous on $ X $.
\end{Th}

\begin{proof} By Remark \ref{rm:nlstms}, $ X $ is a tms, the $ \sigma $-algebra $ \mathscr{A}$ and the measure $ \mu $ are as described in Theorem \ref{t:mtms}. For any open connected set $ E $ in $ X $, we have $\omega(\|.\|,E)=\sup\big\{\big|\|x\|-\|y\|\big|:x,y\in E\big\}\leq\sup\big\{\|x-y\|:x,y\in E\big\}=\text{diam}(E)\leq\mu(E)$. So for any $\varepsilon>0$ and any  $\{E_i\}_{i=1}^\infty\in\pt{}{\varepsilon}$ we have $\displaystyle\sum_{i=1}^{\infty}\omega(\|.\|,E_i)\leq\sum_{i=1}^{\infty}\mu(E_i)<\varepsilon$. Therefore, the norm function is absolutely continuous on $X$.
\end{proof}

\begin{Th}
	Let $ X $ be a tms, $ N,M $ be two normed linear spaces over $ \K $. For any absolutely continuous function $ f:X\to N $ and any bounded linear transformation $ T:N\to M $, the map $ T\circ f:X\to M $ is absolutely continuous.
\end{Th}

\begin{proof}
	For any open connected set $ E $ in $ X $ we have for any $ x,y\in E $, $ |(T\circ f)(x)-(T\circ f)(y)|=|T(f(x)-f(y))|\leq\|T\|\|f(x)-f(y)\| $ and hence $ \omega(T\circ f,E)\leq\|T\|.\omega(f,E) $. This justifies our theorem.
\end{proof}

\noindent\textbf{\underline{Acknowledgement} :} The first author is thankful to UGC, India for financial assistance.

\end{document}